\documentclass[12pt, reqno]{amsart}
\usepackage{amsmath,amssymb,amsthm}
\allowdisplaybreaks
\usepackage{amsmath, amssymb}
\usepackage{amsthm, amsfonts, mathrsfs}
\usepackage{mathptmx}
\usepackage{fullpage}
\usepackage{amsfonts,graphicx}
\numberwithin{equation}{section}
\usepackage[colorlinks=true, pdfstartview=FitV, linkcolor=blue, citecolor=blue, urlcolor=blue]{hyperref}

\newtheorem{theo}{Theorem}[section]
\newtheorem{lemm}[theo]{Lemma}

\newtheorem{coro}[theo]{Corollary}
\newtheorem{prop}[theo]{Proposition}
\newtheorem{rema}[theo]{Remark}
\numberwithin{equation}{section}

\newcommand{\R}{{\mathbb R}}

\newcommand\Z{{\mathbb{Z}}}

\usepackage{xcolor}

\begin{document}
\title{Global regularity and sharp decay rates for the one-dimensional hypo-viscous compressible Navier-Stokes equations in critical Besov spaces}

\author{Chen Liang}

\address{ School of Science,
	Shenzhen Campus of Sun Yat-sen University, Shenzhen 518107, China}

\email{liangch89@mail2.sysu.edu.cn}

\author{Zhaonan Luo}
\address{School of Science,
 Shenzhen Campus of Sun Yat-sen University, Shenzhen 518107, China}

\email{luozhn7@mail.sysu.edu.cn}

\author{Zhaoyang Yin}

\address{ School of Science,
Shenzhen Campus of Sun Yat-sen University, Shenzhen 518107, China}

\email{mcsyzy@mail.sysu.edu.cn }




\begin{abstract}
In this paper, we study the global regularity for the one-dimensional hypo-viscous compressible Navier-Stokes equations in the critical Besov spaces, where the initial data are close to the constant equilibrium state. We introduce an effective velocity, constructed as a combination of the density and the velocity, which is used to capture the high-frequency behavior of $a$. By performing elaborate high-frequency and low-frequency estimates together with a bootstrap argument, we obtain global existence for $\frac{1}{2}\le\beta<1$. In particular, when $\beta=\frac{1}{2}$, we also prove the system is well-posed in critical space. Furthermore, under an additional low-frequency assumption, we establish the optimal time decay rates.

\end{abstract}

\maketitle

\section{Introduction}
In this paper, we consider the following one-dimensional compressible Navier-Stokes (CNS) equations \cite{MR5057982,MR4643428}:
\begin{align}\label{eq0}
\left\{\begin{aligned} 
&\partial_{t}\rho+(\rho u)_{x}=0,\\ 
&\partial_{t}(\rho u)+\mu(-\Delta)^{\beta}u+(\rho u^{2}+P(\rho))_{x}=0,\\
&\rho|_{t=0}=\rho_0,~~u|_{t=0}=u_0,
\end{aligned}\right.
\end{align}
where $\rho$ and $u$ denote the density and the velocity, respectively, $P=\rho^{\gamma}\ (\gamma\ge1)$ is the pressure, and $\mu>0$ is the viscosity coefficient.
  The fractional Laplacian operator $(-\Delta)^{\beta}u$ is defined via the Fourier transform
	$$
	\widehat{(-\Delta)^{\beta}u}=|\xi|^{2\beta}\widehat{u},\quad\quad \widehat{u}(\xi,t)=\int_{\mathbb{R}}e^{-ix\xi}u(x,t) dx.
	$$
Fractional diffusion can precisely characterize physical phenomena such as jump L\'evy processes \cite{MR3185174}, radiation hydrodynamics \cite{MR1031939}, and porous media mechanics \cite{MR2737788,MR2954615}. For simplicity, we set $\mu=1$, $\gamma=1$ and $\rho=a+1$. Then $(a,u)$ takes the form
\begin{align}\label{eq1}
\left\{\begin{aligned} 
&\partial_{t}a+u_x=F,\\ 
&\partial_{t}u+(-\Delta)^{\beta}u+a_x=H,\\
&a|_{t=0}=a_0,~~u|_{t=0}=u_0,
\end{aligned}\right.
\end{align}
where
$$
F=-(au)_x,
$$
and 
$$
H=\underbrace{\left(\frac{a}{1+a}+\frac{P^{\prime}(1)- P^{\prime}(1+a)}{1+a}\right)}_{K(a)}a_x-uu_x+\underbrace{\frac{a}{1+a}}_{G(a)}(-\Delta)^{\beta}u. \\
$$
When $\beta=1$, system (\ref{eq0}) is the classical CNS system. Y. I. Kanel' \cite{MR227619,Kanel1979} solved the Cauchy problem (\ref{eq0}) in Lagrangian coordinates. A. Kazhikhov and V. Shelukhin \cite{KAZHIKHOV1977273} proved global existence and uniqueness in $H^1$. Moreover, A. Kazhikhov \cite{Kazhikhov1979} and S. Kawashima and T. Nishida \cite{MR637519} studied the large-time behavior of system (\ref{eq0}). D. Hoff \cite{DavidHoff} proved global existence when the initial density and velocity belong to $L^2\cap BV$ and $L^2$, respectively. J. Li and Z. Liang \cite{Li2016Some} proved that solutions to the full compressible
Navier--Stokes system converge to equilibrium in $L^p$ for $2<p\le\infty$. However, they did not provide an explicit decay rate. Subsequently, Y. Chen, M. Li, Q. Yao and Z. Yao \cite{MR4549954} obtained the optimal time decay rates in $H^2$ by virtue of spectral analysis and the Fourier splitting method. More results on time decay rates can be found in \cite{MR1671920,MR971685,MR1890882,MR2812966}.

When $0<\beta<1$, Y. Li, P. Qu, Z. Zeng, and D. Zhang \cite{MR5057982} proved that there exist infinitely many weak solutions for the hypo-viscous CNS system on $\mathbb T^d$ with $d\ge2$. S. Wang and S. Zhang \cite{MR4643428} considered the $L^2$ time decay rate of (\ref{eq0}) for the case $\frac{1}{2}<\beta\le1$ in $\mathbb{R}^3$, where the initial data belong to $H^{4}\cap L^{1}$. Moreover, S. Wang and S. Zhang \cite{MR4897629} obtained the time decay rate for the 3D Navier--Stokes-Poisson system with fractional dissipation, which can be regarded as an extension of \cite{MR4643428}. Very recently, C. Liang, Z. Luo and Z. Yin \cite{liang2026} established global regularity and optimal time decay rates for the 2D hypo-viscous CNS system with $\frac{1}{2}\le\beta<1$.

To the best of our knowledge, the global regularity for system (\ref{eq0}) with $\frac{1}{2}\le\beta<1$ has not been studied in the critical Besov spaces. Indeed, there exist critical results for the classical compressible Navier--Stokes equations in high dimensions. In the $L^2$ framework, R. Danchin \cite{MR1855277,MR2354497} proved the local well-posedness when $a_0\in\dot B_{2,1}^{\frac{d}{2}}$ and $u_0\in\dot B_{2,1}^{\frac{d}{2}-1} $, where $d\ge2$ denotes the spatial dimension. Moreover, R. Danchin \cite{MR1779621} obtained global regularity if $a_0\in\dot B_{2,1}^{\frac{d}{2}}\cap\dot B_{2,1}^{\frac{d}{2}-1}$ and $u_0\in\dot B_{2,1}^{\frac{d}{2}-1} $. For the general $L^p$ framework, F. Charve and R. Danchin \cite{MR2679372}, Q. Chen, C. Miao and Z. Zhang \cite{MR2675485}, B. Haspot \cite{MR2847531} proved the global well-posedness for the CNS system if $a_0\in\dot B_{p,1}^{\frac{d}{p}}$, $u_0\in\dot B_{p,1}^{\frac{d}{p}-1}$ and $(a_0,u_0)^l\in\dot B_{2,1}^{\frac{d}{2}-1}$, where $2\le p\le\min(4,\frac{2d}{d-2})$ and $p\neq 4$, if $d=2$. Based on this, our first goal in this paper is to establish global regularity in the critical $L^2$ framework.

Similar to the incompressible Navier--Stokes equations with fractional dissipation, we observe that if $(a,u)$ is a solution of (\ref{eq1}), then
 $$(a_\lambda, u_\lambda)\triangleq\left(a(\lambda^{2\beta}t,\lambda x), \lambda^{2\beta-1}u(\lambda^{2\beta}t,\lambda x)\right)$$ 
 is also a solution of (\ref{eq1}) provided that the function $P$ is replaced by $\lambda^{4\beta-2}P$. As in the case of the incompressible Navier--Stokes equations, a critical space for (\ref{eq1}) should be one whose norm is invariant under $(a,u)\mapsto(a_\lambda,u_\lambda)$, up to a constant independent of $\lambda$. 
The homogeneous Sobolev space $\dot H^{\frac12}\times\dot H^{\frac32-2\beta}$ seems to be a natural candidate. However, we need a $L^\infty$ control of $a$, and the homogeneous Sobolev space is not embedded in $L^\infty$.
Hence, we work in the critical homogeneous Besov space:
$$\dot B_{2,1}^{\frac{1}{2}}\times\dot B_{2,1}^{\frac{3}{2}-2\beta}.$$
For $\frac12\le\beta<1$, the nonlinear terms can be controlled
by Bony's paraproduct estimates since
$\frac32-2\beta>-\frac12$. However, the required product laws fail at the endpoint $\beta=1$. Therefore, the case $\beta=1$ needs to be treated more delicately. Recently, R. Danchin \cite{danchin2026} proved the global regularity and uniqueness for 1D CNS equations with $\beta=1$ by means of the mass Lagrangian coordinates system.

 To further analyze the effects of $(a,u)$, we perform a spectral analysis of the linearized equation. Let us first study the following linear system:
\begin{align}\label{eq2}
\left\{\begin{aligned} 
&\partial_{t}a+u_x=0,\\ 
&\partial_{t}u+(-\Delta)^{\beta}u+a_x=0,\\
&a|_{t=0}=a_0,~~u|_{t=0}=u_0.
\end{aligned}\right.
\end{align}
Taking the Fourier transform of (\ref{eq2}), we have
\begin{align*}
\left\{\begin{aligned} 
&\partial_{t}\widehat{a}+i\xi\widehat{u}=0,\\ 
&\partial_{t}\widehat{u}+|\xi|^{2\beta}\widehat{u}+i\xi\widehat{a}=0.
\end{aligned}\right.
\end{align*}
It is clear that
\begin{align}
\frac{d}{dt}
\begin{pmatrix}
        \widehat{a} \\ \widehat{u}
    \end{pmatrix}
    =
 \begin{pmatrix}
       0 & -i\xi \\ -i\xi& -|\xi|^{2\beta} 
    \end{pmatrix}   
\begin{pmatrix}
        \widehat{a} \\ \widehat{u}
    \end{pmatrix}
    \triangleq L    
    \begin{pmatrix}
        \widehat{a} \\ \widehat{u}
    \end{pmatrix}
    .
\end{align}
The eigenvalues $\lambda_{\pm}$ of $L$ can be computed directly. We consider two distinct cases in the following discussion.

\textbf{Case 1:} $\frac{1}{2}<\beta<1$. In the low-frequency $|\xi|^{4\beta}-4|\xi|^2\le 0$, we have
\begin{align*}
\lambda_{\pm}=\frac{-|\xi|^{2\beta}\pm i\sqrt{4|\xi|^{2}-|\xi|^{4\beta}}}{2}.
\end{align*}
For high frequency, when $|\xi|^{4\beta}-4|\xi|^2>0$, we have
\begin{align*}
\lambda_{\pm}=\frac{-|\xi|^{2\beta}\pm\sqrt{|\xi|^{4\beta}-4|\xi|^{2}}}{2}.
\end{align*}
From this, a parabolic behavior occurs at low frequencies. For the high-frequency part, by virtue of Taylor's expansion, there exists a parabolic effect for $(a,u)$. More precisely, $a$ behaves like the heat kernel operators $e^{-t(-\Delta)^{1-\beta}}$ and $u$ has a parabolic behavior like $e^{-t(-\Delta)^\beta}$. 

\textbf{Case 2:} $\beta=\frac{1}{2}$. It is easy to see 
\begin{align*}
\lambda_{\pm}=\frac{-|\xi|\pm i\sqrt{3}|\xi|}{2}.
\end{align*}
Hence, $(a,u)$ has a parabolic behavior like $e^{-t(-\Delta)^{\frac{1}{2}}}$ in the whole space.

This shows the distinction between the case $\beta=1$ and the case $\frac{1}{2}\le\beta<1$. It follows that $a$ admits a damping effect at high frequencies when $\beta=1$, while $a$ has a parabolic effect at high frequencies when $\frac{1}{2}\le\beta<1$. In fact, this phenomenon has already been observed in our previous study of the CNS equations in Sobolev spaces, and we refer the reader to \cite{liang2026decay,liang2026}. 

In fact, the key to closing the energy is to capture the high-frequency effects. Inspired by \cite{MR2847531}, we introduce an effective velocity
$$w=u+\Lambda^{-2\beta}a_x.$$
The critical coupling helps us obtain the parabolic effect at high frequencies. Moreover, the combination of $a$ and $u$ enables us to deduce an estimate for each individual quantity. After delicate estimates and a bootstrap argument, we obtain the global energy estimate. The method used in this paper may also be useful for proving the existence of solutions to related systems with fractional dissipation.

For $\frac12<\beta<1$, the standard approach to prove uniqueness in higher dimensions cannot be directly applied to the one-dimensional case. Owing to the loss of one derivative caused by the transport equation, following \cite{Bahouri2011}, we need to close the energy estimates for $(\delta a,\delta u)$ in $\dot B^{-1/2}_{2,\infty}\times\dot B^{1/2-2\beta}_{2,\infty}$. Noting $1/2-2\beta<-1/2$,  the standard product laws do not provide the required nonlinear estimates since the regularity index $1/2-2\beta$ is too negative. Even if we reduce the loss of regularity to less than one derivative, the energy estimate does not provide sufficient time integrability to control the term $T_{\delta u}a_{1,x}^h$. Moreover, the nonlocality of the fractional Laplacian makes the analysis in Lagrangian coordinates more delicate (see \cite{LiMiaoXue2023}). Therefore, the uniqueness problem lies beyond the scope of the present paper and deserves further study.


 
We denote by $\mathcal{E}$ the set of functions $(a,u)$ which satisfy
 \begin{align*}
&(a,u)^l\in\widetilde C_{b}\left(\dot B^{\frac{3}{2}-2\beta}_{2,1}\right)\cap \widetilde L^1\left(\dot B^{\frac{3}{2}}_{2,1}\right),\quad a^{h}\in\widetilde C_{b}\left( \dot B^{\frac{1}{2}}_{2,1}\right)\cap \widetilde L^1\left(\dot B^{\frac{5}{2}-2\beta}_{2,1}\right),\\
& u^{h}\in\widetilde C_{b}\left(\dot B^{\frac{3}{2}-2\beta}_{2,1}\right)\cap \widetilde L^1\left(\dot B^{\frac{3}{2}}_{2,1}\right).
 \end{align*}
For the sake of convenience, we denote by $\mathcal{E}_{T}$ the space $\mathcal{E}$ restricted to functions over $[0,T]\times\mathbb{R}$. The precise definitions of these function spaces will be given in Section \ref{section2}. Denote $\Lambda^{s} f=\mathscr{F}^{-1}(|\xi|^s\hat{f})$. In view of the above observations, our first result is stated as follows:
\begin{theo}\label{theo1}
  Let $\frac{1}{2}\le\beta<1$. There exists a small positive constant $c$ such that if $a_0$ is in $\dot B_{2,1}^{\frac{3}{2}-2\beta}\cap\dot B_{2,1}^{\frac{1}{2}}$, $u_0$ is in $\dot B_{2,1}^{\frac{3}{2}-2\beta}$ and
  \begin{align*}
\|(a_0,u_0)\|^l_{\dot B^{\frac{3}{2}-2\beta}_{2,1}}+&\|a_0\|^{h}_{\dot B^{\frac{1}{2}}_{2,1}}+\|u_0\|^{h}_{\dot B^{\frac{3}{2}-2\beta}_{2,1}}\le c,
\end{align*}
then system (\ref{eq1}) admits a global-in-time solution $(a,u)$ belonging to $\mathcal{E}$.
Moreover, when $\frac{1}{2}<\beta<1$, the following energy inequality holds:
\begin{align*}
\frac{d}{dt}\Bigg[\|(a,u)^l\|_{\dot B^{\frac{3}{2}-2\beta}_{2,1}}+&\|(\Lambda^{2\beta-1}a,u)\|^{h}_{\dot B^{\frac{3}{2}-2\beta}_{2,1}}\Bigg]\\\nonumber
&+c_0\left(\|\Lambda a\|^{h}_{\dot B^{\frac{3}{2}-2\beta}_{2,1}}+\|u\|^{h}_{\dot B^{\frac{3}{2}}_{2,1}}+\|(a,u)^l\|_{\dot B^{\frac{3}{2}}_{2,1}}\right)
\le 0,
\end{align*}
where $c_0>0$ is a small enough constant. If $\beta=\frac{1}{2}$, one also has
\begin{align*}
\frac{d}{dt}\|(a,u)\|_{\dot B^{\frac{1}{2}}_{2,1}}+c_0\|(a,u)\|_{\dot B^{\frac{3}{2}}_{2,1}}
\le 0.
\end{align*}
When $\beta=\frac{1}{2}$, system (\ref{eq1}) is globally well-posed in $\dot B^{\frac{1}{2}}_{2,1}$.
\end{theo}
\begin{rema}
Let us emphasize that the well-posedness considered here, in the case $\beta=\frac{1}{2}$, is in the sense of Hadamard, i.e., the system (\ref{eq1}) admits a unique solution that depends continuously on the initial data in some function space $X$. The continuous dependence of solutions for (\ref{eq1}) in the case $\frac{1}{2}<\beta<1$ remains open.
\end{rema}
\begin{rema}
As pointed out by J-Y. Chemin, I. Gallagher and M. Paicu \cite{MR2776367}, Besov spaces with negative indices contain highly oscillatory functions. When $\frac{3}{2}-2\beta<0$, i.e. $\frac{3}{4}<\beta<1$, the Besov space $\dot B^{\frac{3}{2}-2\beta}_{2,1}$ contains highly oscillatory functions; an example in one dimension is
\begin{align*}
u_{\varepsilon}(x)=\sin \left(\frac{x}{\varepsilon}\right)\phi(x),
\end{align*}
where $\phi\in\mathcal{S}(\mathbb{R})$ and $\varepsilon>0$ is small enough. Therefore, our results can be applied to highly oscillatory initial data.
\end{rema}
\begin{rema}
There are several remaining problems. First, the uniqueness in the critical space is still open for $\frac{1}{2}<\beta<1$. Secondly, the existence of local solutions for large initial data in the critical $L^2$ framework for system (\ref{eq1}) is still unknown. Furthermore, can we establish the local and global solvability of system (\ref{eq1}) in the general $L^p$ framework?
\end{rema}
Next, we want to study the time decay rates for (\ref{eq1}). Inspired by the work of \cite{MR3609245,MR4188989,Xu2019,MR3360739}, we additionally assume that the initial data $(a_0,u_0)$ belong to the Besov space $\dot B_{2,\infty}^{-\sigma}$ with $2\beta-\frac{3}{2}<\sigma\le\frac{1}{2}$. In fact, the key to obtaining the optimal time decay rate is to prove the propagation of regularity, that is, we need to show that the solution $(a,u)$ also belongs to $\dot B_{2,\infty}^{-\sigma}$. Then by interpolation between Besov spaces, we can derive a Lyapunov-type inequality. By this method, we obtain the optimal time decay rates. More precisely, our second result is the following:
\begin{theo}\label{theo2}
Let $\frac{1}{2}\le\beta<1$ and $2\beta-\frac{3}{2}<\sigma\le\frac{1}{2}$. Let $(a,u)$ be the global solution under the assumptions of Theorem \ref{theo1}. If additionally $(a_0,u_0)^{l}\in \dot B_{2,\infty}^{-\sigma}$ and
\begin{align*}
\|(a_0,u_0)\|^{l}_{\dot B_{2,\infty}^{-\sigma}}\le C,
\end{align*}
 then for any $-\sigma< s\le \frac{3}{2}-2\beta$, there holds 
\begin{align*}
\|\Lambda^s(a,u)\|_{L^2}\le C(1+t)^{-\frac{s+\sigma}{2\beta}}.
\end{align*}
Furthermore, for any $\frac32-2\beta<s\le\frac12$, we have
\[
\|\Lambda^s a\|_{L^2}\le C(1+t)^{-\frac{s+\sigma}{2\beta}}.
\]
\end{theo}
\begin{rema}
The decay rates obtained in Theorem \ref{theo2} are optimal in the sense that they coincide with those of the corresponding generalized heat equation
\[
\left\{
\begin{aligned}
&\partial_t U+(-\Delta)^\beta U=0,
    \qquad x\in\mathbb{R},\quad t>0,\\
&U(x,0)=U_0\in\dot{B}^{-\sigma}_{2,\infty}.
\end{aligned}
\right.
\]
\end{rema}
\begin{rema}
 When $1\le p\le 2$, $L^{p}\hookrightarrow\dot{B}_{2, \infty}^{-\sigma}$, where $\frac{1}{p}=\sigma+\frac{1}{2}$. It follows that the above results still hold true when additionally $\left(a_{0}, u_{0}\right)\in L^{p}$ with $\frac{1}{p}>2\beta-1$ and
  \begin{align*}
  \|(a_0,u_0)\|_{L^{p}}\le C.
  \end{align*} 
\end{rema}
\begin{rema}
Recently, L. Brandolese, L. Shou, J. Xu and P. Zhang \cite{MR4792291} proved that the $\dot B_{2,\infty}^{-\sigma}$-boundedness condition is sharp for obtaining optimal time decay rates in $\mathbb{R}^d$ with $d\ge 2$. A natural question is whether this conclusion holds in the one-dimensional case with $\frac{1}{2}\le\beta<1$.
\end{rema}

\textbf{Organization of the paper}
The rest of the paper is organized as follows. In Section \ref{section2}, we introduce some lemmas that will be used
in the sequel. In Section \ref{section3}, we present the proof of global regularity for system (\ref{eq1}) with small initial data.
In Section \ref{section4}, we prove the time decay rates for system (\ref{eq1}).
\section{Preliminaries}\label{section2}
We now recall the Littlewood--Paley decomposition theory and Besov spaces.
\begin{lemm}\cite{Bahouri2011}\label{lemma2.1}
Let $\mathcal{C}$ be the annulus $\{\xi\in\mathbb{R}:\frac 3 4\leq|\xi|\leq\frac 8 3\}$. There exist radial functions $\chi$ and $\varphi$, valued in the interval $[0,1]$, belonging respectively to $\mathscr{D}(B(0,\frac 4 3))$ and $\mathscr{D}(\mathcal{C})$, and such that
$$ \forall\xi\in\mathbb{R},\ \chi(\xi)+\sum_{j\geq 0}\varphi(2^{-j}\xi)=1, $$
$$ \forall\xi\in\mathbb{R}\backslash\{0\},\ \sum_{j\in\mathbb{Z}}\varphi(2^{-j}\xi)=1, $$
$$ |j-j'|\geq 2\Rightarrow\mathrm{Supp}\ \varphi(2^{-j}\cdot)\cap \mathrm{Supp}\ \varphi(2^{-j'}\cdot)=\emptyset, $$
$$ j\geq 1\Rightarrow\mathrm{Supp}\ \chi(\cdot)\cap \mathrm{Supp}\ \varphi(2^{-j}\cdot)=\emptyset. $$
The set $\widetilde{\mathcal{C}}=B(0,\frac 2 3)+\mathcal{C}$ is an annulus, and we have
$$ |j-j'|\geq 5\Rightarrow 2^{j}\mathcal{C}\cap 2^{j'}\widetilde{\mathcal{C}}=\emptyset. $$
Further, we have
$$ \forall\xi\in\mathbb{R},\ \frac 1 2\leq\chi^2(\xi)+\sum_{j\geq 0}\varphi^2(2^{-j}\xi)\leq 1, $$
$$ \forall\xi\in\mathbb{R}\backslash\{0\},\ \frac 1 2\leq\sum_{j\in\mathbb{Z}}\varphi^2(2^{-j}\xi)\leq 1. $$
\end{lemm}

$\mathscr{F}$ represents the Fourier transform and its inverse is denoted by $\mathscr{F}^{-1}$.
Let $u$ be a tempered distribution in $\mathcal{S}_{h}'(\mathbb{R})=\mathcal{S}'(\mathbb{R})/P(\mathbb{R})$. For all $j\in\mathbb{Z}$, define
$$
\dot\Delta_j u=\mathscr{F}^{-1}(\varphi(2^{-j}\cdot)\mathscr{F}u),\quad
\dot S_j u=\mathscr{F}^{-1}(\chi(2^{-j}\cdot)\mathscr{F}u).
$$
Then the Littlewood--Paley decomposition is given as follows:
$$ u=\sum_{j\in\mathbb{Z}}\dot\Delta_j u \quad \text{in}\ \mathcal{S}_{h}'(\mathbb{R}). $$
Let $s\in\mathbb{R},\ 1\leq p,r\leq\infty.$ The homogeneous Besov spaces $\dot B^s_{p,r}$ are defined by
$$ \dot B^s_{p,r}=\left\{u\in S'_{h}(\mathbb{R}):\|u\|_{\dot B^s_{p,r}}=\Big\|(2^{js}\|\dot\Delta_j u\|_{L^p})_j \Big\|_{l^r(\mathbb{Z})}<\infty\right\}.$$
For any positive time $T$, the time-space Besov spaces are defined by
\begin{align*}
L^{\rho}_T(\dot B^s_{p,r}) = \left\{u\in L^{\rho}(0,T;S'_{h}(\mathbb{R})):\|u\|_{L^{\rho}_T(\dot B^s_{p,r})} = \left\|\left\|2^{js}\|\dot\Delta_ju\|_{L^p}\right\|_{l^r({\mathbb{Z}})}\right\|_{L^{\rho}_T}<\infty\right\},
\end{align*}
and
\begin{align*}
\widetilde{L}^{\rho}_T(\dot B^s_{p,r}) = \left\{u\in L^{\rho}(0,T;S'_{h}(\mathbb{R})):\|u\|_{\tilde{L}^{\rho}_T(\dot B^s_{p,r})} = \left\|2^{js}\|\dot\Delta_ju\|_{L^{\rho}_T(L^p)}\right\|_{l^r({\mathbb{Z}})}<\infty\right\}.
\end{align*}
For simplicity, when $T=+\infty$, we denote by $L^{\rho}(\dot B^s_{p,r})$ and $\widetilde{L}^{\rho}(\dot B^s_{p,r})$ the spaces $L^{\rho}_{\mathbb{R}^{+}}(\dot B^s_{p,r})$ and $\widetilde{L}^{\rho}_{\mathbb{R}^{+}}(\dot B^s_{p,r})$, respectively. Moreover, the following embedding relationships hold:
$$
L^{\rho}_T(\dot B^s_{p,r})\hookrightarrow \tilde{L}^{\rho}_T(\dot B^s_{p,r}) ~~~~\text{if}~~~~r\geq\rho~~~~\text{and}~~~~
\tilde{L}^{\rho}_T(\dot B^s_{p,r})\hookrightarrow L^{\rho}_T(\dot B^s_{p,r})  ~~~~\text{if}~~~~r\leq\rho.$$
We shall also use the following functional spaces:
\begin{align*}
\widetilde{C}_T(\dot B^s_{p,r}) = \left\{u\in C([0,T];\dot B^s_{p,r}(\mathbb{R})):\|u\|_{\tilde{L}^{\infty}_T(\dot B^s_{p,r})} <\infty\right\}.
\end{align*}
\begin{align*}
\widetilde C_{b}(\dot B_{p,r}^{s})=\left\{u\in C(\mathbb{R}^{+};\dot B_{p,r}^{s}(\mathbb{R})):\|u\|_{\widetilde{L}^{\infty}(\dot B^s_{p,r})}<+\infty\right\}.
\end{align*}
In this paper, we restrict our attention to the case $p=2$. Because of the different behaviors at low and high frequencies, we introduce the following hybrid Besov spaces. Let us fix an integer $j_{0}$ (which will be determined in the proof of the main theorem) and set
$$\left\| u\right\| _{\dot{B}_{2,1}^{s}}^{\ell} \triangleq \sum_{j\leq
j_{0}}2^{js}\| \dot{\Delta}_{j}u\|_{L^{2}} \ \ \mbox{and} \ \ \|u\|_{\dot{B}_{2,1}^{s}}^{h}\triangleq \sum_{j\geq j_{0}-1}2^{js}\| \dot{\Delta}_{j}u\| _{L^{2}},$$
$$\left\|u\right\| _{\tilde{L}_{T}^{\infty} (\dot{B}_{2,1}^{s})}^{\ell} \triangleq
\sum_{j\leq j_{0}}2^{js}\|\dot{\Delta}_{j}u\|_{L_{T}^{\infty} (L^{2})} \ \
\mbox{and} \ \ \|u\| _{\tilde{L}_{T}^{\infty} (\dot{B}_{2,1}^{s})}^{h}\triangleq \sum_{j\geq j_{0}-1}2^{js}\| \dot{\Delta}_{j}u\|
_{L_{T}^{\infty} (L^{2})}.$$
The following inequalities will be used frequently in our proof.
\begin{lemm}\cite{Bahouri2011}\label{lemma2.2}
Let $\mathscr{C}$ be an annulus and $\mathscr{B}$ a ball. There exists a constant $C$ such that for any $k\in\mathbb{N}$, $1\leq p\leq q\leq\infty$, and any function $u\in L^p(\mathbb{R}^d)$, we have
$$
\mathrm{Supp~}\widehat{u}\subset\lambda\mathscr{B}\Longrightarrow\|D^ku\|_{L^q}=\sup_{|\alpha|=k}\|\partial^\alpha u\|_{L^q}\leq C^{k+1}\lambda^{k+d(\frac{1}{p}-\frac{1}{q})}\|u\|_{L^p},
$$
$$
\mathrm{Supp~}\widehat{u}\subset\lambda\mathscr{C}\Longrightarrow C^{-k-1}\lambda^k\|u\|_{L^p}\leq\|D^ku\|_{L^p}\leq C^{k+1}\lambda^k\|u\|_{L^p}.
$$
\end{lemm}
The following classical properties are often used:
\begin{lemm}\cite{Bahouri2011}\label{lemma2.3}
\begin{itemize}
\item \ \emph{Scaling invariance:} For any $s\in \mathbb{R}$ and $r\in
[1,\infty ]$, there exists a constant $C=C(s,r)$ such that for all $\lambda >0$ and $u\in \dot{B}_{2,r}^{s}$,
$$
C^{-1}\lambda ^{s-\frac {1}{2}}\left\|u\right\|_{\dot{B}_{2,r}^{s}}
\leq \left\|u(\lambda \cdot)\right\|_{\dot{B}_{2,r}^{s}}\leq C\lambda ^{s-\frac {1}{2}}\left\|u\right\|_{\dot{B}_{2,r}^{s}}.
$$

\item \ \emph{Completeness:} $\dot{B}^{s}_{2,r}$ is a Banach space whenever $%
s<\frac{1}{2}$ or $s\leq \frac{1}{2}$ and $r=1$.

\item \ \emph{Action of Fourier multipliers:} If $F$ is a smooth function on $\mathbb{R}^{d}\backslash \{0\}$,
homogeneous of degree $m$, then
$$F(D):\dot{B}_{2,r}^{s}\rightarrow \dot{B}_{2,r}^{s-m}.$$

\item  Complex interpolation: Let $1\leq r,r_{1},r_{2}\leq \infty$. if $u\in \dot{B}_{2,r_1}^{s}\cap \dot{B}_{2,r_2}^{\tilde{s}} $ and $s\neq \tilde{s}$, then $u\in \dot{B}_{2,r}^{\theta s+(1-\theta)\tilde{s}}$ for all $\theta\in (0,1)$ and
$$\left\|u\right\|_{\dot{B}_{2,r}^{\theta s+(1-\theta )\tilde{s}}}\leq \left\|u\right\|_{\dot{B}_{2,r_{1}}^{s}}^{\theta} \left\|u\right\|_{\dot{B}_{2,r_2}^{\tilde{s}}}^{1-\theta }$$
with $\frac{1}{r}=\frac{\theta}{r_{1}}+\frac{1-\theta}{r_{2}}$.

\item Real interpolation: if $u\in \dot{B}_{2,\infty}^{s}\cap \dot{B}_{2,\infty}^{\tilde{s}} $ and $s<\tilde{s}$, then $u\in \dot{B}_{2,1}^{\theta s+(1-\theta)\tilde{s}}$ for all $\theta\in (0,1)$ and
$$\left\|u\right\|_{\dot{B}_{2,1}^{\theta s+(1-\theta )\tilde{s}}}\leq \frac{C}{\theta(1-\theta)(\tilde{s}-s)} \left\|u\right\| _{\dot{B}_{2,\infty}^{s}}^{\theta} \left\|u\right\|_{\dot{B}_{2,\infty}^{\tilde{s}}}^{1-\theta }.$$
\end{itemize}
\end{lemm}

Let $u$ and $v$ be tempered distributions in $\mathscr{S}_{h}'$. The homogeneous paraproduct of $v$ by $u$ is defined as follows:
\begin{align*}
	 T_u v = \sum_{j} \dot S_{j-1}u\dot\Delta_j v,
\end{align*}
and the homogeneous remainder of $u$ and $v$ is defined by
\begin{align*}
	 R(u,v) = \sum_{|k-j|\leq1} \dot\Delta_{k} u  \dot\Delta_{j} v\triangleq\sum_{k\in\mathbb{Z}} \dot\Delta_{k}u\tilde{\dot\Delta}_k v.
\end{align*}
At least formally, we obtain the so-called Bony decomposition:
\begin{align*}
	uv= T_u v + T_v u +  R(u,v).
\end{align*}
The following lemmas are concerned with the continuity of the paraproduct operator $T$ and remainder operator $R$.

\begin{lemm}\cite{Bahouri2011}\label{lemma2.4}
	For any $(s,t)\in\mathbb{R}\times(-\infty,0)$ and $(p,p_1,p_2,r,r_1,r_2)\in[1,\infty]^6$, there exists a constant $C$ such that
	\begin{align*}
		\| T_u v\|_{\dot B^s_{p,r}} \leq C^{1+|s|} \|u\|_{L^{p_1}}\|v\|_{\dot B^s_{p_2,r}},
	\end{align*}
	with $(u,v)\in L^{p_1}\times \dot B^s_{p_2,r}$ and $\frac{1}{p}=\frac{1}{p_1}+\frac{1}{p_2}$. Moreover, we have
	\begin{align*}
		\| T_u v\|_{\dot B^{s+t}_{p,r}} \leq \frac{C^{1+|s+t|}}{-t} \|u\|_{\dot B^{t}_{\infty,r_1}}\|v\|_{\dot B^s_{p,r_2}},
	\end{align*}
	with $(u,v) \in \dot B^{t}_{\infty,r_1}\times \dot B^s_{p,r_2}$ and $\frac{1}{r}=\min\{1,\frac{1}{r_1}+\frac{1}{r_2}\}$.
\end{lemm}
\begin{lemm}\cite{Bahouri2011}\label{lemma2.4.2}
There exists a constant $C$ such that the following estimates hold.

Let $(s_1, s_2) \in \mathbb{R}^2$ and $(p_1, p_2, r_1, r_2) \in [1, \infty]^4$.
Assume that
\[
\frac{1}{p} = \frac{1}{p_1} + \frac{1}{p_2} \leq 1
\quad\text{and}\quad
\frac{1}{r}= \frac{1}{r_1} + \frac{1}{r_2} \leq 1.
\]
Then the following estimates hold for the remainder operator $R$:\\
(i): If $s_1 + s_2 > 0$, then for all  $(u, v) \in \dot{B}^{s_1}_{p_1, r_1} \times \dot{B}^{s_2}_{p_2, r_2}$,
\begin{align*}
 \|{R}(u, v)\|_{\dot{B}^{s_1 + s_2}_{p, r}}
\leq \frac{C^{\,|s_1 + s_2| + 1}}{s_1 + s_2}
\|u\|_{\dot{B}^{s_1}_{p_1, r_1}} \|v\|_{\dot{B}^{s_2}_{p_2, r_2}},
\end{align*}
(ii) If  $r = 1$  and  $s_1 + s_2 \geq 0$, then for all $ (u, v) \in \dot{B}^{s_1}_{p_1, r_1} \times \dot{B}^{s_2}_{p_2, r_2}$,
\begin{align*}
 \|R(u, v)\|_{\dot{B}^{s_1 + s_2}_{p, \infty}}
\leq C^{\,|s_1 + s_2| + 1}
\|u\|_{\dot{B}^{s_1}_{p_1, r_1}} \|v\|_{\dot{B}^{s_2}_{p_2, r_2}}.
\end{align*}
\end{lemm}
The following product estimates in Besov spaces play a key role in our analysis of the nonlinear terms of (\ref{eq1}).
\begin{lemm}\label{lemma2.5}\cite{Bahouri2011}
Let $s>0$ and $1\leq p,r\leq\infty$. Then $\dot{B}^{s}_{p,r}\cap L^{\infty}$ is an algebra and
\begin{equation*}
\left\|uv\right\|_{\dot{B}^{s}_{p,r}}\le C\left\|u\right\|_{L^{\infty}}\left\|v\right\|_{\dot{B}^{s}_{p,r}}+\left\|v\right\|_{L^{\infty}}\left\|u\right\|_{\dot{B}^{s}_{p,r}}.
\end{equation*}
Let $-\frac{1}{2}<s_{2}\le\frac{1}{2}$.
Then it holds that
\begin{equation}\label{201}
\left\|uv\right\|_{\dot{B}^{s_{2}}_{2,1}}\le C\left\|u\right\|_{\dot{B}^{\frac{1}{2}}_{2,1}}\left\|v\right\|_{\dot{B}^{s_{2}}_{2,1}}.
\end{equation}
Additionally, for exponents $-\frac{1}{2}<s\le\frac{1}{2}$, one has
\begin{equation}\label{202}
\left\|uv\right\|_{\dot{B}^{-s}_{2,\infty}}\le C\left\|u\right\|_{\dot{B}^{\frac{1}{2}}_{2,1}}\left\|v\right\|_{\dot{B}^{-s}_{2,\infty}}.
\end{equation}
\end{lemm}
The following lemma is concerned with composition estimates and will be used later.
\begin{lemm}\label{lemma2.6}\cite{Bahouri2011,MR3609245}
Let $F:\R\rightarrow\R$ be smooth with $F(0)=0$.
For all $1\leq r\leq\infty$ and $s>0$, it holds that
$F(u)\in \dot {B}^{s}_{2,r}\cap L^{\infty}$ for $u\in \dot {B}^{s}_{2,r}\cap L^{\infty}$,  and
$$\left\|F(u)\right\|_{\dot B^{s}_{2,r}}\leq C\left\|u\right\|_{\dot B^{s}_{2,r}}$$
with $C$ depending only on $\left\|u\right\|_{L^{\infty}}$, $F'$ \emph{(}and higher derivatives\emph{)}, $s$.

In the case $s>-\frac {1}{2}$ then $u\in\dot {B}^{s}_{2,r}\cap\dot {B}^{\frac {1}{2}}_{2,1}$
implies that $F(u)\in \dot {B}^{s}_{2,r}\cap\dot {B}^{\frac {1}{2}}_{2,1}$, and
\begin{equation*}
\left\|F(u)\right\|_{\dot B^{s}_{2,r}}\leq C\left(1+\|u\|_{\dot {B}^{\frac {1}{2}}_{2,1}}\right)\|u\|_{\dot {B}^{s}_{2,r}}.
\end{equation*}
Moreover, the case $s=-\frac{1}{2}$ holds true when $r=\infty$.
 For all $1\leq r\leq\infty$ and $-\frac{1}{2}<s<\frac{1}{2}$, if $f_{i}\in\dot{B}^{s}_{2,r}\cap\dot {B}^{\frac {1}{2}}_{2,1}$ ($i=1, 2$), we have that
$$\|F(f_{1})-F(f_{2})\|_{\dot B^{s}_{2,r}}\leq C\|f_{1}-f_{2}\|_{\dot {B}^{s}_{2,r}},$$
where $C>0$ is a constant depending on $\|f_{i}\|_{\dot B^{\frac{1}{2}}_{2,1}}$, $F$ and $s$. Moreover, the case $s=\frac{1}{2}$ holds true when $r=1$.
\end{lemm}

\begin{lemm}\label{lemma2.7}
Let $-\frac{1}{2}<s\leq\frac {3}{2}$ and $0\le\alpha\le 1$.
There exists a constant $C>0$ depending only on $s$ such that for all $j\in\Z$, it holds that
\begin{equation*}
\|[v\partial_x, \Lambda^{\alpha}\dot{\Delta}_{j}]a\|_{L^{2}}\leq
Cc_{j}2^{-j\left(s-\alpha\right)}\|v_x\|_{\dot{B}^{\frac{1}{2}}_{2,1}}\|a\|_{\dot{B}^{s}_{2,1}},
\end{equation*}
where the commutator
$\left[\cdot,\cdot\right]$ is defined by $[f,g]=fg-gf$ and $(c_{j})_{j\in\Z}$ denotes
a sequence such that $\|(c_{j})\|_{l^{1}}\leq 1$.\\
In the limit case $s=-\frac{1}{2}$, we have
\begin{equation*}
\sup_{j\in\mathbb{Z}}2^{(-\alpha-\frac{1}{2})j}\|[v\partial_x, \Lambda^{\alpha}\dot{\Delta}_{j}]a\|_{L^{2}}\leq
C\|v_x\|_{\dot{B}^{\frac{1}{2}}_{2,1}}\|a\|_{\dot{B}^{-\frac{1}{2}}_{2,\infty}},
\end{equation*}
\end{lemm}
\begin{proof}
The proof is similar to that of Lemma 2.100 in \cite{Bahouri2011}. Using Bony's decomposition for $[v\partial_x, \Lambda^{\alpha}\dot{\Delta}_{j}]a$, we have
\begin{align}\label{203}
[v\partial_x, \Lambda^{\alpha}\dot{\Delta}_{j}]a=[T_{v},\Lambda^{\alpha} \dot{\Delta}_{j}]\partial_x a&+T_{\Lambda^{\alpha}\dot\Delta_j\partial_x a}v-\Lambda^{\alpha}\dot\Delta_jT_{\partial_x a}v\\\nonumber
&+R(v,\partial_x\Lambda^{\alpha}\dot\Delta_j a)-\Lambda^{\alpha}\dot\Delta_j R(v,\partial_x a)\nonumber.
\end{align}
By virtue of Lemma \ref{lemma2.1}, we have 
\begin{align*}
[T_{v},\Lambda^{\alpha} \dot{\Delta}_{j}]\partial_x a=2^{j\alpha}\sum_{|j-j'|\le 4}[\dot S_{j'-1}v,\widetilde{\Delta}_{j}]\partial_x \dot\Delta_{j'}a,
\end{align*}
where $\widetilde{\Delta}_{j}f=(|\xi|^{\alpha}\varphi)(2^{-j}D)f$. Let $g=\mathscr{F}^{-1}(|\xi|^{\alpha}\varphi)$, then we conclude that
\begin{align*}
[\dot S_{j'-1}v,\widetilde{\Delta}_{j}]\partial_x \dot\Delta_{j'}a=2^{j}\int_{\mathbb{R}}g(2^{j}(x-y))\left(S_{j'-1}v(x)-S_{j'-1}v(y)\right)\partial_x \dot\Delta_{j'}a(y)dy.
\end{align*}
By Bernstein's inequality and the mean value theorem, we obtain
\begin{align*}
\|[T_{v},\Lambda^{\alpha} \dot{\Delta}_{j}]\partial_x a\|_{L^2}&\le C\|\partial_x v\|_{L^\infty}\sum_{|j-j'|\le 4}\|\Lambda^{\alpha}\dot\Delta_{j'}a\|_{L^2}\\
&\le C\|\partial_x v\|_{\dot{B}^{\frac{1}{2}}_{2,1}}\sum_{|j-j'|\le 4}\|\Lambda^{\alpha}\dot\Delta_{j'}a\|_{L^2}.
\end{align*}
To handle the second term, we may write 
\begin{align*}
\|T_{\Lambda^{\alpha}\dot\Delta_j\partial_x a}v\|_{L^2}&\le C\sum_{j'\ge j-3}\|\dot S_{j'-1}\Lambda^{\alpha}\partial_x\dot\Delta_ja\dot\Delta_{j'}v\|_{L^2}\\
&\le C\sum_{j'\ge j-3}2^{j-j'}\|\Lambda^{\alpha}\dot\Delta_ja\|_{L^2}\|\dot\Delta_{j'}\partial_x v\|_{L^\infty} \\
&\le C2^{-j(s-\alpha)}\|a\|_{\dot{B}^{s}_{2,1}}c_j\sum_{j'\ge j-3}2^{j-j'}\|\dot\Delta_{j'}\partial_x v\|_{L^\infty} \\
&\le C2^{-j(s-\alpha)}\|a\|_{\dot{B}^{s}_{2,1}}\|v_x\|_{\dot{B}^{\frac{1}{2}}_{2,1}}c_j, 
\end{align*}
where $\displaystyle c_j=\frac{2^{j(s-\alpha)}\|\Lambda^{\alpha}\dot\Delta_ja\|_{L^2}}{\|a\|_{\dot{B}^{s}_{2,1}}}$. If $s<\frac{3}{2}$, using Lemma \ref{lemma2.4} leads to
\begin{align*}
\|\Lambda^{\alpha}T_{\partial_x a}v\|_{\dot{B}^{s-\alpha}_{2,1}}&\le C\|T_{\partial_x a}v\|_{\dot{B}^{s}_{2,1}}\le C\|\partial_x a\|_{\dot{B}^{s-\frac{3}{2}}_{\infty,\infty}}\|v\|_{\dot{B}^{\frac{3}{2}}_{2,1}}\le C\|a\|_{\dot{B}^{s}_{2,1}}\|v\|_{\dot{B}^{\frac{3}{2}}_{2,1}}.
\end{align*}
When $s=\frac{3}{2}$, we have
\begin{align*}
\|\Lambda^{\alpha}T_{\partial_x a}v\|_{\dot{B}^{\frac{3}{2}-\alpha}_{2,1}}&\le C\|T_{\partial_x a}v\|_{\dot{B}^{\frac{3}{2}}_{2,1}}\le C\|\partial_x a\|_{L^{\infty}}\|v\|_{\dot{B}^{\frac{3}{2}}_{2,1}}\le C\|a\|_{\dot{B}^{\frac{3}{2}}_{2,1}}\|v\|_{\dot{B}^{\frac{3}{2}}_{2,1}}.
\end{align*}
For the second last term, we obtain
\begin{align*}
\|R(v,\partial_x\Lambda^{\alpha}\dot\Delta_j a)\|_{L^2}\le  &\ C\sum_{|j-j'|\le1}\|(\dot\Delta_{j'-1}+\dot\Delta_{j'}+\dot\Delta_{j'+1})v\|_{L^\infty}\|\dot\Delta_j\partial_x\Delta_{j'}\Lambda^{\alpha} a\|_{L^2}\\
\le&\ C\|\partial_x v\|_{L^\infty}2^{j(\alpha-s)}\sum_{|j-j'|\le1}2^{(j'-j)(\alpha-s)}2^{j's}\|\dot\Delta_{j'} a\|_{L^2}\\
\le&\ C2^{j(\alpha-s)}\|v_x\|_{\dot{B}^{\frac{1}{2}}_{2,1}}\|a\|_{\dot{B}^{s}_{2,1}}c_j. 
\end{align*}
Thanks to $s>-\frac{1}{2}$ and Lemma \ref{lemma2.4.2}, we arrive at 
\begin{align*}
\|\Lambda^{\alpha} R(v,\partial_x a)\|_{\dot{B}^{s-\alpha}_{2,1}}\le &\ C\| R(v,\partial_x a)\|_{\dot{B}^{s}_{2,1}}\le\ C\|R(v,\partial_x a)\|_{\dot{B}^{s+\frac{1}{2}}_{1,1}}\le \ C\|v\|_{\dot{B}^{\frac{3}{2}}_{2,1}}\|\partial_x a\|_{\dot{B}^{s-1}_{2,1}}.
\end{align*}
Inserting the above estimates into (\ref{203}) yields the desired result. The limit case follows by a similar argument. This completes the proof of Lemma \ref{lemma2.7}.
\end{proof}
\section{Global existence and uniqueness to the CNS equations}\label{section3}
In this section, we consider the global regularity for the CNS system (\ref{eq1}). We first consider the case $\frac{1}{2}<\beta<1$. 
First, we derive energy estimates for the low and high frequencies.
\begin{lemm}\label{lemma3.1}
Let $\frac{1}{2}<\beta<1$. Let $(a,u)\in\mathcal{E}_{T}$ be a local strong solution of system (\ref{eq1}) under the assumptions of Theorem \ref{theo1}. Then we have
\begin{align*}
&\frac{d}{dt}\left[\|(a,u)^l\|_{\dot B^{\frac{3}{2}-2\beta}_{2,1}}+\|(\Lambda^{2\beta-1}a,u)\|^{h}_{\dot B^{\frac{3}{2}-2\beta}_{2,1}}\right]+c_0\left[\|\Lambda a\|^{h}_{\dot B^{\frac{3}{2}-2\beta}_{2,1}}+\|u\|^{h}_{\dot B^{\frac{3}{2}}_{2,1}}+\|(a,u)^l\|_{\dot B^{\frac{3}{2}}_{2,1}}\right]\\
\le&\ C\|u_x\|_{\dot B^{\frac{1}{2}}_{2,1}}\|a\|_{\dot B^{\frac{1}{2}}_{2,1}}+C\|F\|^{h}_{\dot B^{\frac{5}{2}-4\beta}_{2,1}}+C\|H\|^{h}_{\dot B^{\frac{3}{2}-2\beta}_{2,1}}+C\|(F,H)^{l}\|_{\dot B^{\frac{3}{2}-2\beta}_{2,1}},
\end{align*}
where $c_0>0$ is a small enough constant. 
\end{lemm}
\begin{proof}The proof consists of two steps.\\
\textbf{Step 1: Low-frequency estimates}\\
Let $f^l=\dot S_{j_{0}}f$ and $f^h=f-f^l$. Taking $\dot\Delta_{j}\dot S_{j_{0}}$ to the system (\ref{eq1}), we have
\begin{align}\label{301}
\left\{\begin{aligned} 
&\partial_{t}a_j^{l}+u_{j,x}^{l}=F_j^{l},\\ 
&\partial_{t}u_j^{l}+(-\Delta)^{\beta}u^{l}_j+a_{j,x}^{l}=H_j^{l}.
\end{aligned}\right.
\end{align}
Taking the standard $L^2$ inner product, we have
\begin{align}\label{302}
\frac{1}{2}\frac{d}{dt}\|(a_j^{l},u_j^{l})\|_{L^2}^{2}+\|\Lambda^{\beta}u_j^{l}\|_{L^{2}}^{2}=\langle F_j^{l},a_j^{l}\rangle+\langle H_j^{l},u_j^{l}\rangle,
\end{align}
\begin{align}\label{303}
\frac{d}{dt}\langle\Lambda^{2\beta-2}a_{j,x}^{l},u_j^{l}\rangle+&\|\Lambda^{\beta}a_j^{l}\|_{L^{2}}^{2}=\langle \Lambda^{2\beta-2}F_{j,x}^{l},u_j^{l}\rangle+\langle\Lambda^{2\beta-2} H_j^{l},a_{j,x}^{l}\rangle\\\nonumber
&+\|\Lambda^{\beta}u_j^{l}\|_{L^{2}}^{2}-\langle\Lambda^{2\beta-2} (-\Delta)^{\beta}u^{l}_j,a_{j,x}^{l}\rangle,
\end{align}
and
\begin{align}\label{304}
\frac{1}{2}\frac{d}{dt}\|\Lambda^{2\beta-2}a_{j,x}^{l}\|_{L^2}^{2}=\langle \Lambda^{2\beta-2}F_{j,x}^{l},\Lambda^{2\beta-2}a_{j,x}^{l}\rangle-\langle \Lambda^{2\beta-2}u_{j,xx}^{l},\Lambda^{2\beta-2}a_{j,x}^{l}\rangle,
\end{align}
Combining above inequalities, we have
\begin{align}
\frac{1}{2}\frac{d}{dt}&\underbrace{\left[2\|(a_j^{l},u_j^{l})\|_{L^2}^{2}+2\langle\Lambda^{2\beta-2}a_{j,x}^{l},u_j^{l}\rangle+\|\Lambda^{2\beta-2}a_{j,x}^{l}\|_{L^2}^{2}\right]}_{L_{j}^{2}}+
\|\Lambda^{\beta}(a_j^{l},u_j^{l})\|_{L^{2}}^{2}\\\nonumber
=&2\langle F_j^{l},a_j^{l}\rangle+2\langle H_j^{l},u_j^{l}\rangle+\langle \Lambda^{2\beta-2}F_{j,x}^{l},u_j^{l}\rangle+\langle\Lambda^{2\beta-2} H_j^{l},a_{j,x}^{l}\rangle+\langle \Lambda^{2\beta-2}F_{j,x}^{l},\Lambda^{2\beta-2}a_{j,x}^{l}\rangle.
\end{align}
Since $j\le j_0$, by Lemma \ref{lemma2.2} it is easy to check
\begin{align*}
L_{j}^{2}\approx \|(a_j^{l},u_j^{l})\|_{L^2}^{2},\quad \|\Lambda^{\beta}(a_j^{l},u_j^{l})\|_{L^{2}}^{2}\approx 2^{2j\beta}\|(a_j^{l},u_j^{l})\|_{L^2}^{2}.
\end{align*}
We end up with 
\begin{align*}
\frac{1}{2}\frac{d}{dt}L_{j}^{2}+c_{0}2^{2j\beta}L_{j}^{2}\le CL_j\|(F_{j}^{l},H_{j}^{l})\|_{L^2},
\end{align*}
which gives rise to
\begin{align}\label{306}
\frac{d}{dt}L_{j}+c_{0}2^{2j\beta}L_{j}\le C\|(F_{j}^{l},H_{j}^{l})\|_{L^2}.
\end{align}
Multiplying (\ref{306}) by $2^{j(\frac{3}{2}-2\beta)}$ and taking the $l^{1}$ norm, we have
\begin{align}\label{307}
\frac{d}{dt}\|(a,u)^l\|_{\dot B^{\frac{3}{2}-2\beta}_{2,1}}+c_{0}\|(a,u)^l\|_{\dot B^{\frac{3}{2}}_{2,1}}\le C\|(F,H)^{l}\|_{\dot B^{\frac{3}{2}-2\beta}_{2,1}}.
\end{align}
\textbf{Step 2: High-frequency estimates}

Let $w=u+\Lambda^{-2\beta}a_x$. It is easy to check that $w$ satisfies
\begin{align}\label{308}
\frac{d}{dt}w+(-\Delta)^{\beta}w=\Lambda^{2-2\beta}w+\Lambda^{-2\beta}F_x+H-\Lambda^{2-4\beta}a_x.
\end{align}
Applying $\dot\Delta_{j}$ to (\ref{308}) and taking the standard $L^2$ inner product, we obtain
\begin{align*}
\frac{1}{2}\frac{d}{dt}\|w_j\|_{L^2}^{2}+c2^{2j\beta}\|w_j\|_{L^2}^{2}\le \left(\|\Lambda^{-2\beta}F_{j,x}+H_j\|_{L^2}+\|\Lambda^{2-2\beta}w_j-\Lambda^{2-4\beta}a_{j,x}\|_{L^2}\right)\|w_j\|_{L^2},
\end{align*}
which gives rise to
\begin{align}\label{3010}
\frac{1}{2}\frac{d}{dt}\|w_j\|_{L^2}+c2^{2j\beta}\|w_j\|_{L^2}\le \|\Lambda^{-2\beta}F_{j,x}+H_j\|_{L^2}+\|\Lambda^{2-2\beta}w_j-\Lambda^{2-4\beta}a_{j,x}\|_{L^2}.
\end{align}
Noting $a$ fulfills
\begin{align}\label{310}
\frac{d}{dt}a+(au)_x+\Lambda^{2-2\beta}a=-w_x.
\end{align}
Applying the operator $\Lambda^{2\beta-1}\dot\Delta_{j}$ to (\ref{310}), we obtain
\begin{align}\label{311}
\frac{d}{dt}\Lambda^{2\beta-1}a_j+u\Lambda^{2\beta-1}a_{j,x}+\Lambda a_{j}=-\Lambda^{2\beta-1}w_{x,j}-\Lambda^{2\beta-1}(au_x)_{j}+[u\partial_x,\dot\Delta_{j}\Lambda^{2\beta-1}]a.
\end{align}
Taking the $L^2$ inner product in (\ref{311}), we have
\begin{align*}
\frac{1}{2}\frac{d}{dt}\|\Lambda^{2\beta-1}a_j\|_{L^2}^{2}+\|\Lambda^{\beta} a_j\|_{L^2}^{2}=&\frac{1}{2}\int_{\mathbb{R}}u_{x}(\Lambda^{2\beta-1}a_{j})^{2}dx-\int_{\mathbb{R}}\Lambda ^{2\beta-1}w_{x,j}\Lambda^{2\beta-1}a_{j}dx\\\nonumber
&-\int_{\mathbb{R}}\left(\Lambda^{2\beta-1}(au_x)_{j}-[u\partial_x,\dot\Delta_{j}\Lambda^{2\beta-1}]a\right)\Lambda^{2\beta-1}a_jdx.
\end{align*}
By virtue of Hölder's inequality, we obtain
\begin{align*}
\frac{1}{2}\frac{d}{dt}\|\Lambda^{2\beta-1}a_j\|_{L^2}^{2}+\|\Lambda^{\beta} a_j\|_{L^2}^{2}\le& C\|u_x\|_{L^{\infty}}\|\Lambda^{2\beta-1}a_{j}\|_{L^{2}}^{2}+\|\Lambda ^{2\beta-1}w_{x,j}\|_{L^2}\|\Lambda^{2\beta-1}a_{j}\|_{L^{2}}\\\nonumber
&+\left(\|\Lambda^{2\beta-1}(au_x)_{j}\|_{L^2}+\|[u\partial_x,\dot\Delta_{j}\Lambda^{2\beta-1}]a\|_{L^{2}}\right)\|\Lambda^{2\beta-1}a_j\|_{L^2}.
\end{align*}
Then we achieve
\begin{align}\label{314}
\frac{1}{2}\frac{d}{dt}\|\Lambda^{2\beta-1}a_j\|_{L^2}+c\|\Lambda a_j\|_{L^2}\le& C\|u_x\|_{L^{\infty}}\|\Lambda^{2\beta-1}a_{j}\|_{L^{2}}+\|\Lambda ^{2\beta-1}w_{x,j}\|_{L^2}\\\nonumber
&+\|\Lambda^{2\beta-1}(au_x)_{j}\|_{L^2}+\|[u\partial_x,\dot\Delta_{j}\Lambda^{2\beta-1}]a\|_{L^{2}}.
\end{align}
Combining (\ref{3010}) and (\ref{314}), we obtain
\begin{align}\label{3141}
&\frac{1}{2}\frac{d}{dt}\left[\|w_j\|_{L^2}+\epsilon\|\Lambda^{2\beta-1}a_j\|_{L^2}\right]+c\epsilon\|\Lambda a_j\|_{L^2}+c2^{2j\beta}\|w_j\|_{L^2}\\\nonumber
\le& C\epsilon\|u_x\|_{L^{\infty}}\|\Lambda^{2\beta-1}a_{j}\|_{L^{2}}+\epsilon\|\Lambda ^{2\beta-1}w_{x,j}\|_{L^2}\\\nonumber
&+\epsilon\|\Lambda^{2\beta-1}(au_x)_{j}\|_{L^2}+\epsilon\|[u\partial_x,\dot\Delta_{j}\Lambda^{2\beta-1}]a\|_{L^{2}}\\\nonumber
&+\|\Lambda^{-2\beta}F_{j,x}+H_j\|_{L^2}+\|\Lambda^{2-2\beta}w_j-\Lambda^{2-4\beta}a_{j,x}\|_{L^2}.
\end{align}
Thanks to Lemma \ref{lemma2.2} and the fact $\frac{1}{2}<\beta<1$, we see that
\begin{align*}
\|\Lambda^{2-4\beta}a_{j,x}\|_{L^2}\le C2^{j(2-4\beta)}\|\Lambda a_{j}\|_{L^2}\le C2^{j_{0}(2-4\beta)}\|\Lambda a_{j}\|_{L^2},\quad j\ge j_0-1,
\end{align*}
and
\begin{align*}
 \|\Lambda^{2-2\beta}w_{j}\|_{L^2}\le C2^{j(2-2\beta)}\| w_{j}\|_{L^2}\le C2^{j_{0}(2-4\beta)}2^{2j\beta}\| w_{j}\|_{L^2},\quad j\ge j_0-1.   
\end{align*}
Taking $\epsilon$ small enough and $j_0$ large enough in (\ref{3141}), we have
\begin{align}\label{315}
&\frac{d}{dt}\left[\|w_j\|_{L^2}+\|\Lambda^{2\beta-1}a_j\|_{L^2}\right]+c_0\|\Lambda a_j\|_{L^2}+c_02^{2j\beta}\|w_j\|_{L^2}\\\nonumber
\le& C\epsilon\|u_x\|_{L^{\infty}}\|\Lambda^{2\beta-1}a_{j}\|_{L^{2}}+\epsilon\|\Lambda^{2\beta-1}(au_x)_{j}\|_{L^2}+\epsilon\|[u\partial_x,\dot\Delta_{j}\Lambda^{2\beta-1}]a\|_{L^{2}}\\\nonumber
&+2^{(1-2\beta)j}\|F_{j}\|_{L^2}+\|H_j\|_{L^2}.
\end{align}
Recall
\begin{align*}
u=w-\Lambda^{-2\beta}a_x.
\end{align*}
As $\frac{1}{2}<\beta<1$, we conclude that
\begin{align}\label{316}
&\frac{d}{dt}\left[\|u_j\|_{L^2}+\|\Lambda^{2\beta-1}a_j\|_{L^2}\right]+c_0\|\Lambda a_j\|_{L^2}+c_02^{2j\beta}\|u_j\|_{L^2}\\\nonumber
\le& C\epsilon\|u_x\|_{L^{\infty}}\|\Lambda^{2\beta-1}a_{j}\|_{L^{2}}+\epsilon\|\Lambda^{2\beta-1}(au_x)_{j}\|_{L^2}+\epsilon\|[u\partial_x,\dot\Delta_{j}\Lambda^{2\beta-1}]a\|_{L^{2}}\\\nonumber
&+2^{-2j\beta+j}\|F_{j}\|_{L^2}+\|H_j\|_{L^2}.
\end{align}
Multiplying (\ref{316}) by $2^{(\frac{3}{2}-2\beta)j}$ and taking the $l^1$ norm for $j\ge j_0-1$, we have
\begin{align}\label{317}
&\frac{d}{dt}\|(\Lambda^{2\beta-1}a,u)\|^{h}_{\dot B^{\frac{3}{2}-2\beta}_{2,1}}+c_0\|\Lambda a\|^{h}_{\dot B^{\frac{3}{2}-2\beta}_{2,1}}+c_0\|u\|^{h}_{\dot B^{\frac{3}{2}}_{2,1}}\\\nonumber
\le& C\epsilon\|u_x\|_{\dot B^{\frac{1}{2}}_{2,1}}\|a\|_{\dot B^{\frac{1}{2}}_{2,1}}+\epsilon\|[u\partial_x,\dot\Delta_{j}\Lambda^{2\beta-1}]a\|^{h}_{\dot B^{\frac{3}{2}-2\beta}_{2,1}}+\|F\|^{h}_{\dot B^{\frac{5}{2}-4\beta}_{2,1}}+\|H\|^{h}_{\dot B^{\frac{3}{2}-2\beta}_{2,1}}.
\end{align}
 Combining (\ref{307}) and (\ref{317}) and taking $\alpha=2\beta-1$ and $s=\frac{1}{2}$ in Lemma \ref{lemma2.7}, it shows that
\begin{align}\label{318}
&\frac{d}{dt}\left[\|(a,u)^l\|_{\dot B^{\frac{3}{2}-2\beta}_{2,1}}+\|(\Lambda^{2\beta-1}a,u)\|^{h}_{\dot B^{\frac{3}{2}-2\beta}_{2,1}}\right]+c_{0}\left(\|\Lambda a\|^{h}_{\dot B^{\frac{3}{2}-2\beta}_{2,1}}+\|u\|^{h}_{\dot B^{\frac{3}{2}}_{2,1}}+\|(a,u)^l\|_{\dot B^{\frac{3}{2}}_{2,1}}\right)\\\nonumber
\le& C\epsilon\|u_x\|_{\dot B^{\frac{1}{2}}_{2,1}}\|a\|_{\dot B^{\frac{1}{2}}_{2,1}}+\epsilon\|[u\partial_x,\dot\Delta_{j}\Lambda^{2\beta-1}]a\|^{h}_{\dot B^{\frac{3}{2}-2\beta}_{2,1}}+\|F\|^{h}_{\dot B^{\frac{5}{2}-4\beta}_{2,1}}+\|H\|^{h}_{\dot B^{\frac{3}{2}-2\beta}_{2,1}}+C\|(F,H)^{l}\|_{\dot B^{\frac{3}{2}-2\beta}_{2,1}}.\\\nonumber
\le& C\epsilon\|u_x\|_{\dot B^{\frac{1}{2}}_{2,1}}\|a\|_{\dot B^{\frac{1}{2}}_{2,1}}+\|F\|^{h}_{\dot B^{\frac{5}{2}-4\beta}_{2,1}}+\|H\|^{h}_{\dot B^{\frac{3}{2}-2\beta}_{2,1}}+C\|(F,H)^{l}\|_{\dot B^{\frac{3}{2}-2\beta}_{2,1}}.
\end{align}
This completes the proof of Lemma \ref{lemma3.1}.
\end{proof}

We are now in a position to derive the global energy estimate.
\begin{lemm}\label{lemma3.2}
Let $\frac{1}{2}<\beta<1$. Let $(a,u)\in\mathcal{E}_{T}$ be a local strong solution of system (\ref{eq1}) under the assumptions of Theorem \ref{theo1}. Then we have
\begin{align*}
\frac{d}{dt}\Bigg[\|(a,u)^l\|_{\dot B^{\frac{3}{2}-2\beta}_{2,1}}+&\|(\Lambda^{2\beta-1}a,u)\|^{h}_{\dot B^{\frac{3}{2}-2\beta}_{2,1}}\Bigg]\\\nonumber
&+c_0\left(\|\Lambda a\|^{h}_{\dot B^{\frac{3}{2}-2\beta}_{2,1}}+\|u\|^{h}_{\dot B^{\frac{3}{2}}_{2,1}}+\|(a,u)^l\|_{\dot B^{\frac{3}{2}}_{2,1}}\right)
\le 0.
\end{align*}
\end{lemm}

\begin{proof}
Let
\begin{align*}
E(t)&=\|(a,u)^l\|_{\dot B^{\frac{3}{2}-2\beta}_{2,1}}+\|(\Lambda^{2\beta-1}a,u)\|^{h}_{\dot B^{\frac{3}{2}-2\beta}_{2,1}},\\
D(t)&=\|\Lambda a\|^{h}_{\dot B^{\frac{3}{2}-2\beta}_{2,1}}+\|u\|^{h}_{\dot B^{\frac{3}{2}}_{2,1}}+\|(a,u)^l\|_{\dot B^{\frac{3}{2}}_{2,1}}.
\end{align*}
Before proceeding further, we assume a priori that
\begin{align*}
E(t)\le c,
\end{align*}
where $c>0$ is a small enough constant.

In order to close the energy estimate, we now estimate the nonlinear terms. Taking $s_2=\frac{3}{2}-2\beta$ in Lemma \ref{lemma2.5}, we obtain the following product laws
\begin{align}\label{3181}
\dot B^{\frac{3}{2}-2\beta}_{2,1}\times\dot B^{\frac{1}{2}}_{2,1}\to\dot B^{\frac{3}{2}-2\beta}_{2,1},
\end{align}
 which will be used frequently. We first deal with the high-frequency case. When $\frac{1}{2}<\beta<\frac{3}{4}$, we decompose $(au)_{x}=au_x+a^{l}_xu^{l}+a^{l}_xu^{h}+a^{h}_xu$. We can easily deduce from Lemma \ref{lemma2.5} that
\begin{align}\label{31811}
\|au_x\|^{h}_{\dot B^{\frac{5}{2}-4\beta}_{2,1}}\le C\|au_x\|^{h}_{\dot B^{\frac{1}{2}}_{2,1}}\le C\|a\|_{\dot B^{\frac{1}{2}}_{2,1}}\|u_x\|_{\dot B^{\frac{1}{2}}_{2,1}}\le cCD,
\end{align}
and
\begin{align}\label{31812}
\|a_x^lu^l\|^{h}_{\dot B^{\frac{5}{2}-4\beta}_{2,1}}\le C\|a_x^lu^l\|^{h}_{\dot B^{\frac{1}{2}}_{2,1}}\le C\|a_x^l\|_{\dot B^{\frac{1}{2}}_{2,1}}\|u^l\|_{\dot B^{\frac{1}{2}}_{2,1}}\le cCD.
\end{align}
Thanks to $\frac{1}{2}<\beta<\frac{3}{4}$, by Lemma \ref{lemma2.4} and Lemma \ref{lemma2.4.2}, we have
\begin{align}\label{31813}
\|a_x^hu\|^{h}_{\dot B^{\frac{5}{2}-4\beta}_{2,1}}\le C\|a_x^h\|_{\dot B^{\frac{3}{2}-2\beta}_{2,1}}\|u\|_{\dot B^{\frac{3}{2}-2\beta}_{2,1}}\le cCD,
\end{align}
where we use the following product law
\begin{align}\label{318140}
\dot B^{\frac{3}{2}-2\beta}_{2,1}\times\dot B^{\frac{3}{2}-2\beta}_{2,1}\to\dot B^{\frac{5}{2}-4\beta}_{2,1}.
\end{align}
Taking $s_2=\frac{5}{2}-4\beta$ in (\ref{201}), we obtain
\begin{align}\label{31814}
\|a_x^lu^h\|^{h}_{\dot B^{\frac{5}{2}-4\beta}_{2,1}}\le C\|a_x^l\|_{\dot B^{\frac{1}{2}}_{2,1}}\|u^h\|_{\dot B^{\frac{5}{2}-4\beta}_{2,1}}\le cCD.
\end{align}
Combining from (\ref{31811}) to (\ref{31814}) yields
\begin{align}\label{3191}
\|F\|^{h}_{\dot B^{\frac{5}{2}-4\beta}_{2,1}} \le cCD.
\end{align}
When $\frac{3}{4}\le\beta<1$, we notice that $-\frac{1}{2}<\frac{7}{2}-4\beta\le\frac{1}{2}$. Then by Lemma \ref{lemma2.5}, we have the following product law:
\begin{align}\label{31910}
\dot B^{\frac{7}{2}-4\beta}_{2,1}\times\dot B^{\frac{1}{2}}_{2,1}\to\dot B^{\frac{7}{2}-4\beta}_{2,1}.
\end{align}
Hence, we write $au=a^hu+a^lu^h+a^lu^l$. Using the above product laws, we obtain
\begin{align}\label{31815}
\|a^h u\|^{h}_{\dot B^{\frac{7}{2}-4\beta}_{2,1}}\le C\|a^h\|_{\dot B^{\frac{1}{2}}_{2,1}}\left(\|u^l\|_{\dot B^{\frac{7}{2}-4\beta}_{2,1}}+\|u^h\|_{\dot B^{\frac{7}{2}-4\beta}_{2,1}}\right)\le cCD,
\end{align}
and
\begin{align}\label{31816}
\|a^l u^h\|^{h}_{\dot B^{\frac{7}{2}-4\beta}_{2,1}}\le C\|a^l\|_{\dot B^{\frac{1}{2}}_{2,1}}\|u^h\|_{\dot B^{\frac{7}{2}-4\beta}_{2,1}}\le cCD,
\end{align}
For the last term, using Lemma \ref{lemma2.5} again, we obtain
\begin{align}\label{31817}
\|a^l u^l\|^{h}_{\dot B^{\frac{7}{2}-4\beta}_{2,1}}\le C\|a^l u^l\|^{h}_{\dot B^{\frac{3}{2}}_{2,1}}\le C\|a^l\|_{\dot B^{\frac{3}{2}}_{2,1}}\|u^l\|_{L^{\infty}}+C\|u^l\|_{\dot B^{\frac{3}{2}}_{2,1}}\|a^l\|_{L^{\infty}}\le cCD,
\end{align}
Therefore, by (\ref{31815})-(\ref{31817}), we conclude that when $\frac{3}{4}\le\beta<1$,
\begin{align}\label{3192}
\|F\|^{h}_{\dot B^{\frac{5}{2}-4\beta}_{2,1}}\le cCD.
\end{align}
Recall $H=K(a)a_x-uu_x+\frac{a}{1+a}(-\Delta)^{\beta}u$, we first estimate the term $K(a)a_x$, 
 it follows from (\ref{3181}) that
 \begin{align}\label{3201}
\|K(a)a_x\|^{h}_{\dot B^{\frac{3}{2}-2\beta}_{2,1}}&\le \|K(a)a_x^{l}\|_{\dot B^{\frac{3}{2}-2\beta}_{2,1}}+\|K(a)a_x^{h}\|_{\dot B^{\frac{3}{2}-2\beta}_{2,1}}\\\nonumber
&\le C\|a\|_{\dot B^{\frac{3}{2}-2\beta}_{2,1}}\|a_x^l\|_{\dot B^{\frac{1}{2}}_{2,1}}+C\|a\|_{\dot B^{\frac{1}{2}}_{2,1}}\|a^h_x\|_{\dot B^{\frac{3}{2}-2\beta}_{2,1}}\\\nonumber
&\le cCD.
 \end{align}
The remaining two terms can be treated as follows
\begin{align}\label{320}
\left\|-uu_x+G(a)(-\Delta)^{\beta}u\right\|^{h}_{\dot B^{\frac{3}{2}-2\beta}_{2,1}}\le C\|u\|_{\dot B^{\frac{3}{2}-2\beta}_{2,1}}\|u\|_{\dot B^{\frac{3}{2}}_{2,1}}+C\|a\|_{\dot B^{\frac{1}{2}}_{2,1}}\|u\|_{\dot B^{\frac{3}{2}}_{2,1}}\le cCD,
\end{align}
Therefore, by (\ref{3201}) and (\ref{320}),  we arrive at 
\begin{align}\label{321}
\|H\|^{h}_{\dot B^{\frac{3}{2}-2\beta}_{2,1}}\le cCD.
\end{align}
Now we are in a position to estimate the nonlinear terms at low frequencies. For $F$, we decompose $(au)_{x}=(au^l)_x+(au^h)_x$. Lemma \ref{lemma2.5} ensures that
\begin{align*}
\|(au^h)_x^{l}\|_{\dot B^{\frac{3}{2}-2\beta}_{2,1}}\le C\|au^h\|_{\dot B^{\frac{1}{2}}_{2,1}}\le C\|a\|_{\dot B^{\frac{1}{2}}_{2,1}}\|u^h\|_{\dot B^{\frac{1}{2}}_{2,1}}\le cCD.
\end{align*}
and due to (\ref{3181}) again, we have
\begin{align*}
\|(au^l)_x^l\|_{\dot B^{\frac{3}{2}-2\beta}_{2,1}}&\le C\left(\|(a_xu^l)^l\|_{\dot B^{\frac{3}{2}-2\beta}_{2,1}}+\|(au_x^l)^l\|_{\dot B^{\frac{3}{2}-2\beta}_{2,1}}\right)\\
&\le C\left(\|a^l_x\|_{\dot B^{\frac{1}{2}}_{2,1}}\|u^l\|_{\dot B^{\frac{3}{2}-2\beta}_{2,1}}+\|a^h_x\|_{\dot B^{\frac{3}{2}-2\beta}_{2,1}}\|u^l\|_{\dot B^{\frac{1}{2}}_{2,1}}+\|a\|_{\dot B^{\frac{3}{2}-2\beta}_{2,1}}\|u\|_{\dot B^{\frac{3}{2}}_{2,1}}\right)\\
&\le cCD.
\end{align*}
Hence, we obtain
\begin{align}\label{322}
\|F^{l}\|_{\dot B^{\frac{3}{2}-2\beta}_{2,1}}\le cCD.
\end{align}
For the nonlinear term $H$, according to (\ref{3201}) and (\ref{320}), we conclude that  
\begin{align}\label{323}
\|H^{l}\|_{\dot B^{\frac{3}{2}-2\beta}_{2,1}}\le cCD.
\end{align}
Combining the estimates (\ref{3191}), (\ref{3192}),  (\ref{321}), (\ref{322}) and (\ref{323}), we conclude that
\begin{align}\label{324}
\frac{d}{dt}\Bigg[\|(a,u)^l\|_{\dot B^{\frac{3}{2}-2\beta}_{2,1}}+&\|(\Lambda^{2\beta-1}a,u)\|^{h}_{\dot B^{\frac{3}{2}-2\beta}_{2,1}}\Bigg]\\\nonumber
&+c_{0}\left(\|\Lambda a\|^{h}_{\dot B^{\frac{3}{2}-2\beta}_{2,1}}+\|u\|^{h}_{\dot B^{\frac{3}{2}}_{2,1}}+\|(a,u)^l\|_{\dot B^{\frac{3}{2}}_{2,1}}\right)
\le 0.
\end{align}
A standard bootstrap argument then yields Lemma \ref{lemma3.2}.
\end{proof}

Next, we turn to the case $\beta=\frac{1}{2}$. We have the following proposition.
\begin{prop}\label{prop3.4}
Let $\beta=\frac{1}{2}$. Let $(a,u)\in\mathcal{E}_{T}$ be a local strong solution of system (\ref{eq1}) under the assumptions of Theorem \ref{theo1}.   
We have the following energy inequality:
\begin{align*}
\frac{d}{dt}\|(a,u)\|_{\dot B^{\frac{1}{2}}_{2,1}}+c_0\|(a,u)\|_{\dot B^{\frac{3}{2}}_{2,1}}
\le 0.
\end{align*}
where $c_0>0$ is a small enough constant. Furthermore, the system (\ref{eq1}) is globally well-posed in $\dot B^{\frac{1}{2}}_{2,1}$.
\end{prop}

\begin{proof}
We also set
\[
E(t)=\|(a,u)\|_{\dot B^{1/2}_{2,1}},
\qquad
D(t)=\|(a,u)\|_{\dot B^{3/2}_{2,1}}.
\]
 Applying $\dot\Delta_{j}$ to system (\ref{eq1}), we have
\begin{align}\label{31801}
\left\{\begin{aligned} 
&\partial_{t}a_j+u_{j,x}=F_j,\\ 
&\partial_{t}u_j+(-\Delta)^{\frac{1}{2}}u_j+a_{j,x}=H_{j}.
\end{aligned}\right.
\end{align}
Taking the standard $L^2$ inner product in (\ref{31801}), we have
\begin{align*}
\frac{1}{2}\frac{d}{dt}\|(a_j,u_j)\|_{L^2}^{2}+\|\Lambda^{\frac{1}{2}}u_j\|_{L^{2}}^{2}=&\langle F_j,a_j\rangle+\langle H_j,u_j\rangle,\\
\frac{d}{dt}\langle\Lambda^{-1}a_{j,x},u_j\rangle+\|\Lambda^{\frac{1}{2}}a_j\|_{L^{2}}^{2}=&\langle \Lambda^{-1}F_{j,x},u_j\rangle+\langle\Lambda^{-1} H_j,a_{j,x}\rangle\\\nonumber
&+\|\Lambda^{\frac{1}{2}}u_j\|_{L^{2}}^{2}-\langle u_j,a_{j,x}\rangle,
\end{align*}
and
\begin{align*}
\frac{1}{2}\frac{d}{dt}\|\Lambda^{-1}a_{j,x}\|_{L^2}^{2}=\langle \Lambda^{-1}F_{j,x},\Lambda^{-1}a_{j,x}\rangle-\langle \Lambda^{-1}u_{j,xx},\Lambda^{-1}a_{j,x}\rangle.
\end{align*}
Combining above inequalities, we have
\begin{align}
\frac{1}{2}\frac{d}{dt}&\underbrace{\left[2\|(a_j,u_j)\|_{L^2}^{2}+2\langle\Lambda^{-1}a_{j,x},u_j\rangle+\|\Lambda^{-1}a_{j,x}\|_{L^2}^{2}\right]}_{L_{j}^{2}}+
\|\Lambda^{\frac{1}{2}}(a_j,u_j)\|_{L^{2}}^{2}\\\nonumber
=&2\langle F_j,a_j\rangle+2\langle H_j,u_j\rangle+\langle \Lambda^{-1}F_{j,x},u_j\rangle+\langle\Lambda^{-1} H_j,a_{j,x}\rangle+\langle \Lambda^{-1}F_{j,x},\Lambda^{-1}a_{j,x}\rangle.
\end{align}
By Lemma \ref{lemma2.2}, it is easy to check
\begin{align*}
L_{j}^{2}\approx \|(a_j,u_j)\|_{L^2}^{2},\quad \|\Lambda^{\frac{1}{2}}(a_j,u_j)\|_{L^{2}}^{2}\approx 2^{j}\|(a_j,u_j)\|_{L^2}^{2}.
\end{align*}
We end up with 
\begin{align*}
\frac{1}{2}\frac{d}{dt}L_{j}^{2}+c_{0}2^{j}L_{j}^{2}\le CL_j\|(F_{j},H_{j})\|_{L^2},
\end{align*}
which gives rise to
\begin{align}\label{31805}
\frac{d}{dt}L_{j}+c_{0}2^{j}L_{j}\le C\|(F_{j},H_{j})\|_{L^2}.
\end{align}
Multiplying (\ref{31805}) by $2^{\frac{1}{2}j}$ and taking the $l^{1}$ norm, we have
\begin{align}\label{31806}
\frac{d}{dt}\|(a,u)\|_{\dot B^{\frac{1}{2}}_{2,1}}+c_{0}\|(a,u)\|_{\dot B^{\frac{3}{2}}_{2,1}}\le C\|(F,H)\|_{\dot B^{\frac{1}{2}}_{2,1}}.
\end{align}
We next deal with the nonlinear terms. For $F$, since $\dot B_{2,1}^{\frac{1}{2}}$ is an algebra, we obtain
\begin{align*}
\|(au)_x\|_{\dot B_{2,1}^{\frac{1}{2}}}\le C\|a\|_{\dot B_{2,1}^{\frac{1}{2}}}\|u_x\|_{\dot B_{2,1}^{\frac{1}{2}}}+C\|u\|_{\dot B_{2,1}^{\frac{1}{2}}}\|a_x\|_{\dot B_{2,1}^{\frac{1}{2}}}\le cCD.
\end{align*}
For $H$, we have
\begin{align*}
\|K(a)a_x\|_{\dot B_{2,1}^{\frac{1}{2}}}&+\|uu_x\|_{\dot B_{2,1}^{\frac{1}{2}}}+\|G(a)\Lambda u\|_{\dot B_{2,1}^{\frac{1}{2}}}\\
&\le C\|a\|_{\dot B_{2,1}^{\frac{1}{2}}}\|a_x\|_{\dot B_{2,1}^{\frac{1}{2}}}
+ C\|u\|_{\dot B_{2,1}^{\frac{1}{2}}}\|u_x\|_{\dot B_{2,1}^{\frac{1}{2}}}
 +C\|a\|_{\dot B_{2,1}^{\frac{1}{2}}}\|\Lambda u\|_{\dot B_{2,1}^{\frac{1}{2}}}\le cCD.
\end{align*}
Inserting the above estimates into (\ref{31806}), we obtain
\begin{align}\label{31807}
\frac{d}{dt}\|(a,u)\|_{\dot B^{\frac{1}{2}}_{2,1}}+c_0\|(a,u)\|_{\dot B^{\frac{3}{2}}_{2,1}}
\le 0.
\end{align}
In this part, we consider the continuous dependence for system (\ref{eq1}) with $\beta=\frac{1}{2}$. As before, let $(\delta a,\delta u)=(a_1-a_2,u_1-u_2)$.
It is clear that $(\delta a,\delta u)$ satisfies the following equations:
\begin{align}\label{eq6}
\left\{\begin{aligned} 
&\partial_{t}\delta a+\delta u_x=\Psi,\\ 
&\partial_{t}\delta u+(-\Delta)^{\frac{1}{2}}\delta u+\delta a_x=\Phi,
\end{aligned}\right.
\end{align}
where 
\begin{align*}
\Psi=&-(u_{2}\delta a)_x-(a_1\delta u)_x,\\
\Phi=&(K(a_1)-K(a_2))a_{2,x}+K(a_1)\delta a_x-u_{2,x}\delta u-u_1\delta u_x\\
&+(G(a_1)-G(a_2))(-\Delta)^{\frac{1}{2}}u_{2}+G(a_1)(-\Delta)^{\frac{1}{2}}\delta u.
\end{align*}
Let
\begin{align*}
\widetilde E(t)=\|(\delta a,\delta u)\|_{\dot B^{\frac{1}{2}}_{2,1}},\quad
\widetilde D(t)=\|(\delta a,\delta u)\|_{\dot B^{\frac{3}{2}}_{2,1}}.
\end{align*}
Standard energy estimates imply that
\begin{align}\label{350}
\frac{d}{dt}\|(\delta a,\delta u)\|_{\dot B^{\frac{1}{2}}_{2,1}}+\|(\delta a,\delta u)\|_{\dot B^{\frac{3}{2}}_{2,1}}\le C\|(\Psi,\Phi)\|_{\dot B^{\frac{1}{2}}_{2,1}}.
\end{align}
For $\Psi$, note $\dot B^{\frac{1}{2}}_{2,1} $ is an algebra. We then have
\begin{align*}
\|(u_{2}\delta a)_x\|_{\dot B^{\frac{1}{2}}_{2,1}}&\le C\|u_{2,x}\delta a\|_{\dot B^{\frac{1}{2}}_{2,1}}+C\|u_{2}\delta a_x\|_{\dot B^{\frac{1}{2}}_{2,1}}\\
&\le C\|\delta a\|_{\dot B^{\frac{1}{2}}_{2,1}}\|u_{2}\|_{\dot B^{\frac{3}{2}}_{2,1}}+C\|\delta a\|_{\dot B^{\frac{3}{2}}_{2,1}}\|u_{2}\|_{\dot B^{\frac{1}{2}}_{2,1}}\\
&\le C\widetilde E D+C\widetilde D E.
\end{align*}
Similarly, we derive that
\begin{align*}
\|(a_{1}\delta u)_x\|_{\dot B^{\frac{1}{2}}_{2,1}}&\le C\|a_{1}\delta u_x\|_{\dot B^{\frac{1}{2}}_{2,1}}+C\|a_{1,x}\delta u\|_{\dot B^{\frac{1}{2}}_{2,1}}\\
&\le C\|a_1\|_{\dot B^{\frac{1}{2}}_{2,1}}\|\delta u\|_{\dot B^{\frac{3}{2}}_{2,1}}+C\|a_1\|_{\dot B^{\frac{3}{2}}_{2,1}}\|\delta u\|_{\dot B^{\frac{1}{2}}_{2,1}}\\
&\le C\widetilde D E+C\widetilde E D.
\end{align*}
Combining above inequalities, we obtain
\begin{align}\label{351}
\|\Psi\|_{\dot B^{\frac{1}{2}}_{2,1}}\le C\widetilde D E+C\widetilde E D.
\end{align}
The remaining terms can be estimated using the fact that $\dot B^{\frac{1}{2}}_{2,1} $ is an algebra and Lemma \ref{lemma2.6}. We handle the nonlinear term $(K(a_2)-K(a_1))a_{2,x}$ as an example:
\begin{align*}
\|(K(a_1)-K(a_2))a_{2,x}\|_{\dot B^{\frac{1}{2}}_{2,1}}&\le C\|a_{2,x}\|_{\dot B^{\frac{1}{2}}_{2,1}}\|\delta a\|_{\dot B^{\frac{1}{2}}_{2,1}}\le C\widetilde E D,
\end{align*}
The other terms can be bounded in the same way. Finally, we conclude that
\begin{align}\label{352}
\|\Phi\|_{\dot B^{\frac{1}{2}}_{2,1}}&\le C\widetilde E D+C\widetilde D E.
\end{align}
Substituting (\ref{351}) and (\ref{352}) into (\ref{350}), it is clear that
\begin{align*}
\frac{d}{dt}\|(\delta a,\delta u)\|_{\dot B^{\frac{1}{2}}_{2,1}}+c_0\|(\delta a,\delta u)\|_{\dot B^{\frac{3}{2}}_{2,1}}\le C\widetilde E D+C\widetilde D E,
\end{align*}
which results in
\begin{align*}
\frac{d}{dt}\|(\delta a,\delta u)\|_{\dot B^{\frac{1}{2}}_{2,1}}+c_0\|(\delta a,\delta u)\|_{\dot B^{\frac{3}{2}}_{2,1}}\le C\widetilde E D,
\end{align*}
By virtue of Gronwall's inequality again, we thus have
\begin{align*}
\sup_{0\le t\le T}\widetilde E(t)+c_0\int_{0}^{T}\widetilde D(t)dt\le C\exp \left(C\int_{0}^{T}D(t)dt\right) \widetilde E(0)\le  C\widetilde E(0).
\end{align*}
We complete the proof of continuous dependence for system (\ref{eq1}) when $\beta=\frac{1}{2}$.
\end{proof}
\begin{coro}\label{coro3.1}
With slight modifications to the proofs of Lemmas \ref{lemma3.1}-\ref{lemma3.2} and Proposition \ref{prop3.4}, we can prove that
\begin{align*}
&(a,u)^l\in\widetilde L^\infty\left(\dot B^{\frac{3}{2}-2\beta}_{2,1}\right)\cap \widetilde L^1\left(\dot B^{\frac{3}{2}}_{2,1}\right),\quad 
a^{h}\in \widetilde L^\infty\left(\dot B^{\frac{1}{2}}_{2,1}\right)\cap \widetilde L^1\left(\dot B^{\frac{5}{2}-2\beta}_{2,1}\right),\\
\quad 
&u^{h}\in\widetilde L^\infty\left(\dot B^{\frac{3}{2}-2\beta}_{2,1}\right)\cap \widetilde L^1\left(\dot B^{\frac{3}{2}}_{2,1}\right).
\end{align*}
\end{coro}
\textbf{Proof of Theorem \ref{theo1}:}

We briefly sketch the construction of a local strong solution to system (\ref{eq1}), since it is similar to that of Theorem 10.2 in \cite{Bahouri2011}. We only consider the case $\frac{1}{2}<\beta<1$. The endpoint case $\beta=\frac12$ follows from a standard contraction argument based on the estimates in Proposition \ref{prop3.4}. Let $L_n^2$ be the set of $L^2$ functions spectrally supported in the annulus $C_n=\left\{\frac{1}{n}\le|\xi|\le n\right\}$. The space $L_n^2$ is endowed with the standard $L^2$ topology. Let $\Omega_n$ be the set of functions $(a,u) \in (L_n^2)^{2}$ satisfying $\inf_{x \in \mathbb{R}} a(x) > -1$. Define $\mathscr{F}{\mathbb{E}}_n U(\xi) = \mathbf{1}_{C_n}(\xi) \mathscr{F} U(\xi)$. Let $u_L$ be the solution of the following linear heat equation
\begin{align*}
\partial_t u+(-\Delta)^{\beta} u=0,\quad u|_{t=0}=u_0.
\end{align*}
Let $u_L^n=\mathbb{E}_n u_L$. Let $(a^n,\bar u^n)$ be the solutions of the following system
\begin{align}\label{eq7}
\frac{d}{dt}
\begin{pmatrix}
        a^n \\ \bar{u}^n
    \end{pmatrix}
    =
    \begin{pmatrix}
        F_n(a^n,\bar{u}^n) \\ H_n(a^n,\bar{u}^n)
    \end{pmatrix}
    ,
    \qquad
\begin{pmatrix}
a^n \\
\bar{u}^n
\end{pmatrix}\bigg|_{t=0}
=
\begin{pmatrix}
{\mathbb{E}}_n a_0 \\
0
\end{pmatrix}.
\end{align}
where 
$$F_n(a^n,\bar{u}^n)=-\mathbb{E}_n [((1+a^n)u^n)_x],$$ 
and 
$$H_n(a^n, \bar{u}^n) = -{\mathbb{E}}_n(-\Delta)^{\beta} \bar{u}^n-{\mathbb{E}}_n(u^nu^n_x)+\mathbb{E}_n(G(a^n)(-\Delta)^{\beta}u^n) - {\mathbb{E}}_n \left(\frac{P'(1+a^n)}{1+a^n}a^n_x\right),$$ 
and
$$
u^n=\mathbb{E}_n u_L+\bar{u}^n.
$$
Then we have $u^n=\bar{u}^n+u_L^n$. The Cauchy–Lipschitz theorem implies that (\ref{eq7}) admits a unique maximal solution $(a^n,\bar{u}^n)$ in the space $C^1([0,T_n^*[;\Omega_n)$.

It is not difficult to check that the solution $(a^n,u^n)$ satisfies the global energy estimate (\ref{324}) and $T_n^*=\infty$. Hence, for any $T>0$, $\{u^n\}$ is uniformly bounded in
$\widetilde L^\infty_{T} \left(\dot B^{\frac{3}{2}-2\beta}_{2,1}\right)\cap \widetilde L^1_{T}\left(\dot B^{\frac{3}{2}}_{2,1}\right)$. Interpolation inequality implies that $\{u^n\}_{n\ge1}$ is uniformly bounded in $\widetilde L^{r}_{T}\left(\dot B^{\frac{3}{2}-2\beta+\frac{2\beta}{r}}_{2,1}\right)$ for any $1\le r\le+\infty$. We also deduce that  $\{a^{l,n}\}_{n\ge1}$ is uniformly bounded in $\widetilde L^\infty_{T}\left(\dot B^{\frac{3}{2}-2\beta}_{2,1}\right)\cap \widetilde L^{1}_{T}\left(\dot B^{\frac{3}{2}}_{2,1}\right)$ and $\{a^{h,n}\}_{n\ge1}$ is uniformly bounded in $\widetilde L^\infty_{T}\left(\dot B^{\frac{1}{2}}_{2,1}\right)\cap \widetilde L^{1}_{T}\left(\dot B^{\frac{5}{2}-2\beta}_{2,1}\right)$. Let $\bar{a}^n=a^n-{\mathbb{E}}_n a_0$. It is easy to see $\bar{a}^n$ satisfies
\begin{align*}
\partial_t \bar{a}^n=-{\mathbb{E}}_n[((1+a^n)u^n)_x].
\end{align*}
By virtue of Lemma \ref{lemma2.5}, we have
\begin{align*}
\|(u^n+a^nu^n)_x\|_{\dot B_{2,1}^{-\frac{1}{2}}}\le C\|a^{n}u^n\|_{\dot B_{2,1}^{\frac{1}{2}}}+C\|u^n\|_{\dot B_{2,1}^{\frac{1}{2}}}\le C\left(1+\|a^{n}\|_{\dot B_{2,1}^{\frac{1}{2}}}\right)\|u^n\|_{\dot B_{2,1}^{\frac{1}{2}}}.
\end{align*}
By using the uniform bound for $\{a^n\}_{n\ge1}$  and $\{u^n\}_{n\ge1}$, we readily derive from the above estimate that $\{\partial_t \bar{a}^n\}_{n\ge1}$ is uniformly bounded in $L_{T}^{\frac{2\beta}{2\beta-1}}\left(\dot B_{2,1}^{-\frac{1}{2}}\right)$. As regards $\bar{u}^n$, it satisfies the following
\begin{align*}
\partial_t \bar{u}^n = -{\mathbb{E}}_n \left( u^nu^n_x - G(a^n) (-\Delta)^{\beta} u^n +(-\Delta)^{\beta} \bar{u}^n + \left(\frac{P'(1+a^n)}{1+a^n}a_{x}^{n}\right) \right).
\end{align*}
By virtue of Lemmas \ref{lemma2.5} and \ref{lemma2.6} again, we have 
\begin{align*}
\|u^nu_{x}^{n}\|_{\dot B_{2,\infty}^{-\frac{1}{2}}}
&\le C\|u^{n}\|_{\dot B_{2,1}^{\frac{1}{2}}}^{2},
\end{align*}
and
\begin{align*}
\|G(a^n) (-\Delta)^{\beta} u^n\|_{\dot B_{2,\infty}^{-\frac{1}{2}}}
\le C\|u^{n}\|_{\dot B_{2,1}^{2\beta-\frac{1}{2}}}\|a^n\|_{\dot B_{2,1}^{\frac{1}{2}}}.
\end{align*}
Using Lemma \ref{lemma2.3}, we deduce
\begin{align*}
\|(-\Delta)^{\beta} \bar{u}^n\|_{\dot B_{2,\infty}^{-\frac{1}{2}}}
\le C\|\bar{u}^{n}\|_{\dot B_{2,1}^{2\beta-\frac{1}{2}}}.
\end{align*}
For the last term, thanks to Lemma \ref{lemma2.6}, we deduce 
\begin{align*}
\left\|\left(\frac{P'(1+a^n)}{1+a^n}a_{x}^{n}\right)\right\|_{\dot B_{2,\infty}^{-\frac{1}{2}}}=\|(A(a^n))_{x}\|_{\dot B_{2,\infty}^{-\frac{1}{2}}}
&\le C\|a^{n}\|_{\dot B_{2,1}^{\frac{1}{2}}},
\end{align*}
where $A$ is a smooth function vanishing at $0$. Combining the above estimates and the uniform boundedness of $\{a^n\}_{n\ge1}$ and $\{u^n\}_{n\ge1}$, we deduce that $\{\partial_t \bar{u}^n\}_{n\ge1}$ is uniformly bounded in $L_{T}^{\frac{\beta}{2\beta-1}}\left(\dot B_{2,\infty}^{-\frac{1}{2}}\right)$. Hence, we have the following conclusion:
\begin{align*}
&\{\bar{a}^n\}_{n\ge1} \quad\text{is uniformly bounded in}\quad L^\infty_{T}\left(\dot B^{\frac{1}{2}}_{2,1}\right)\cap L^\infty_{T}\left(\dot B^{\frac{3}{2}-2\beta}_{2,1}\right)\cap C_{T}^{\frac{1}{2\beta}}\left(\dot B_{2,1}^{-\frac{1}{2}}\right),\\\nonumber
&\{\bar{u}^n\}_{n\ge1} \quad\text{is uniformly bounded in}\quad L^\infty_{T}\left(\dot B^{\frac{3}{2}-2\beta}_{2,1}\right)\cap L^1_{T}\left(\dot B^{\frac{3}{2}}_{2,1}\right)\cap C_{T}^{\frac{1-\beta}{\beta}}\left(\dot B_{2,\infty}^{-\frac{1}{2}}\right).\nonumber
\end{align*}
Let $(\phi_j)_{j\in\mathbb{N}}$ be a sequence of smooth functions with values in $[0,1]$, supported in the ball $B(0,j+1)$ and equal to 1 on $B(0,j)$. The above conclusion implies that
\begin{align*}
(\phi_j\bar{a}^n,\phi_j\bar{u}^n) \quad\text{is equicontinuous in}\quad  C_{T}\left(B_{2,1}^{-\frac{1}{2}}\times  B_{2,\infty}^{-\frac{1}{2}}\right).
\end{align*}
Hence, Ascoli's theorem and Cantor diagonal process ensure that there exist functions $(\bar{a},\bar{u})$ such that
\begin{align*}
(\phi_j\bar{a}^n,\phi_j\bar{u}^n) \longrightarrow(\phi_j\bar{a},\phi_j\bar{u}) \quad \text{in} \quad  C_{T}\left( B_{2,1}^{-\frac{1}{2}}\times  B_{2,\infty}^{-\frac{1}{2}}\right).
\end{align*}
Using Besov space interpolation, we have
\begin{align*}
&\phi_j\bar{a}^n \longrightarrow \phi_j\bar{a} \quad \text{in} \quad  C_{T}\left(B_{2,1}^{-\frac{1}{2}+\varepsilon}\right),\quad 0\le\varepsilon<1.\\
&\phi_j\bar{u}^n \longrightarrow   \phi_j\bar{u} \quad \text{in} \quad  C_{T}\left(   B_{2,1}^{-\frac{1}{2}+\varepsilon}\right),\quad 0<\varepsilon<2-2\beta.
\end{align*}
It is not difficult to check that $\bar u$ is in $\widetilde L^{1}_{T}\left(\dot B^{\frac{3}{2}}_{2,1}\right)$. From the definition of $(a^n,u^n)$, we see that
\begin{align*}
&\phi_j{a}^n \longrightarrow  \phi_j{a} \quad \text{in} \quad  C_{T}\left(B_{2,1}^{-\frac{1}{2}+\varepsilon}\right),\quad 0\le\varepsilon<1.\\
&\phi_j{u}^n \longrightarrow  \phi_j{u} \quad \text{in} \quad  C_{T}\left(   B_{2,1}^{-\frac{1}{2}+\varepsilon}\right),\quad 0<\varepsilon<2-2\beta,
\end{align*}
where $(a,u)=(a_0+\bar{a},u_L+\bar{u})$. It remains to check that $(a,u)$ solves (\ref{eq1}). We take $\mathbb{E}_n [((1+a^n)u^n)_x]$ and $\mathbb{E}_n(G(a^n)(-\Delta)^{\beta}u^n)$ as examples to analyze. Fix some $\varphi\in D(\mathbb{R})$, taking $p$ big enough, such that $\phi_p\equiv 1$ in supp$\varphi$. We first write
\begin{align*}
\Big\langle\mathbb{E}_n [((1+a^n)u^n)_x]&-((1+a)u)_{x},\varphi\Big\rangle=-\Big\langle ((1+a^n)u^n),(\mathbb{E}_n-\mathrm{Id})\varphi_{x}\Big\rangle\\
&-\Big\langle ((1+\phi_p a^n)\phi_p u^n)-(1+\phi_p a)\phi_p u,\varphi_x\Big\rangle.
\end{align*}
Combining the uniform bounds for $a^n$ and $u^n$, we conclude that  
\begin{align*}
&\int_{0}^{T}\Big\langle ((1+a^n)u^n),(\mathbb{E}_n-\mathrm{Id})\varphi_x\Big\rangle dt\\
\le&\ C\left(1+\|a^n\|_{L_{T}^\infty\left(\dot B^{\frac{1}{2}}_{2,1}\right)}\right)\|u^n\|_{L_{T}^\infty\left(\dot B^{\frac{3}{2}-2\beta}_{2,1}\right)}\|(\mathbb{E}_n-\mathrm{Id})\varphi_x\|_{\dot B^{2\beta-\frac{3}{2}}_{2,1}}\to0,\ (n\to\infty).
\end{align*}
For the other terms, we have
\begin{align*}
&\int_{0}^{T}\Big\langle (\phi_p a^n-\phi_p a)\phi_p u^n,\varphi_x\Big\rangle dt\\
\le&\ C\|\phi_p a^n-\phi_p a\|_{L_{T}^{2}\left( B^{\frac{3}{2}-2\beta}_{2,1}\right)}\|\phi_p u^n\|_{L_{T}^{2}\left( B^{\frac{1}{2}}_{2,1}\right)}\|\varphi_x\|_{ B^{\beta-\frac{1}{2}}_{2,1}}\to0,\ (n\to\infty),
\end{align*}
and
\begin{align*}
&\int_{0}^{T}\Big\langle (\phi_p u^n-\phi_p u)\phi_p a,\varphi_x\Big\rangle dt\\
\le&\ C\|\phi_p u^n-\phi_p u\|_{L_{T}^2\left( B^{\frac{1}{2}}_{2,1}\right)}\|\phi_p a\|_{L_{T}^2\left( B^{\frac{3}{2}-2\beta}_{2,1}\right)}\|\varphi_x\|_{ B^{2\beta-\frac{3}{2}}_{2,1}}\to0,\ (n\to\infty).
\end{align*}
Combining the above, we derive that
\begin{align*}
\int_{0}^{T}\Big\langle\mathbb{E}_n [((1+a^n)u^n)_x]&-((1+a)u)_{x},\varphi\Big\rangle dt\to0,\ (n\to\infty).
\end{align*}
Proceeding in the same way, we write
\begin{align*}
\Big\langle\mathbb{E}_n [G(a^n)(-\Delta)^{\beta}u^n]&-G(a)(-\Delta)^{\beta}u,\varphi\Big\rangle=\Big\langle G(a^n)(-\Delta)^{\beta}u^n,(\mathbb{E}_n-\mathrm{Id})\varphi\Big\rangle\\
&+\Big\langle G(\phi_p a^n)(-\Delta)^{\beta} u^n-G(\phi_p a)(-\Delta)^{\beta} u,\varphi\Big\rangle.
\end{align*}
Combining the uniform bounds for $(a^n,u^n)$ and Lemma \ref{lemma2.6}, we have
\begin{align*}
-&\int_{0}^{T}\Big\langle G(a^n)(-\Delta)^{\beta}u^n,(\mathbb{E}_n-\mathrm{Id})\varphi\Big\rangle dt\\
\le&\ C\left(1+\|a^n\|_{L_{T}^\infty\left(\dot B^{\frac{1}{2}}_{2,1}\right)}\right)\|u^n\|_{L_{T}^1\left(\dot B^{\frac{3}{2}}_{2,1}\right)}\|(\mathbb{E}_n-\mathrm{Id})\varphi\|_{\dot B^{2\beta-\frac{3}{2}}_{2,1}}\to0,\ (n\to\infty).
\end{align*}
Next, we estimate the second term. Applying Lemmas \ref{lemma2.4}-\ref{lemma2.4.2} again, for any $0<\varepsilon<\min\{2-2\beta,\frac{1}{2}\}$, there holds
\begin{align*}
\int_{0}^{T}&\Big\langle (G(\phi_p a^n)-G(\phi_p a))(-\Delta)^{\beta} u^n,\varphi\Big\rangle dt\\
&\le C\|\phi_p a^n-\phi_p a\|_{L_{T}^\infty\left(B^{\frac{1}{2}-\varepsilon}_{2,\infty}\right)}\| u^n\|_{L_{T}^{1}\left(\dot B^{2\beta-\frac{1}{2}+\varepsilon}_{2,1}\right)}\|\varphi\|_{B^{\frac{1}{2}}_{2,1}}\to0,\ (n\to\infty).
\end{align*}
Note that
\begin{align*}
\|\varphi G(\phi_{p}a)\|_{L_{T}^{\frac{\beta}{1-\beta}}\left(B^{\frac{1}{2}}_{2,1}\right)}\le C\|\phi_p a\|_{L_{T}^\infty\left( B^{\frac{1}{2}}_{2,1}\right)}\|\varphi\|_{ B^{\frac{1}{2}}_{2,1}}.
\end{align*}
It turns out from the above fact and Lemmas \ref{lemma2.4}-\ref{lemma2.4.2} and \ref{lemma2.6} that
\begin{align*}
\int_{0}^{T}&\Big\langle (-\Delta)^{\beta}(u^n-u)G(\phi_p a),\varphi\Big\rangle dt\to 0,\ (n\to\infty),
\end{align*}
where we use the fact that $\{u^{n}\}_{n\ge1}$ converges weakly to $u$ in $L_{T}^{\frac{\beta}{2\beta-1}}\left(\dot B_{2,2}^{2\beta-\frac{1}{2}}\right)$. Therefore, we arrive at
\begin{align*}
\int_{0}^{T}\Big\langle\mathbb{E}_n [G(a^n)(-\Delta)^{\beta}u^n]&-G(a)(-\Delta)^{\beta}u,\varphi\Big\rangle dt\to0,\ (n\to\infty).
\end{align*}
The other terms can be handled in a similar way. Hence, we prove that $(a,u)$ is a solution of system (\ref{eq1}) in the sense of distributions. Thanks to Fatou's lemma and (\ref{324}), we derive that
  \begin{align*}
a\in\widetilde L_{T}^\infty\left(\dot B^{\frac{3}{2}-2\beta}_{2,1}\right)\cap \widetilde L_{T}^\infty\left(\dot B^{\frac{1}{2}}_{2,1}\right),\quad u\in\widetilde L_{T}^\infty\left(\dot B^{\frac{3}{2}-2\beta}_{2,1}\right)\cap \widetilde L_{T}^1\left(\dot B^{\frac{3}{2}}_{2,1}\right).
\end{align*}
Finally, we observe
\begin{align*}
\partial_{t}a+ua_{x}\in L^{1}_{T}\left(\dot B_{2,1}^{\frac{3}{2}-2\beta}\right)\cap L^{1}_{T}\left(\dot B_{2,1}^{\frac{1}{2}}\right),\quad \partial_{t}u\in L_{T}^{1}\left(\dot B_{2,1}^{\frac{3}{2}-2\beta}\right),
\end{align*}
which ensures that $a\in C_{T}\left(\dot B_{2,1}^{\frac{3}{2}-2\beta}\right)\cap C_{T}\left(\dot B_{2,1}^{\frac{1}{2}}\right)$ and $u\in C_{T}\left(\dot B_{2,1}^{\frac{3}{2}-2\beta}\right)$.
Thanks to Lemmas \ref{lemma3.1}-\ref{lemma3.2}, Proposition \ref{prop3.4} and Corollary \ref{coro3.1}, this completes the proof of Theorem \ref{theo1}.
\hfill\qedsymbol

\section{Time decay rates to the CNS equations}\label{section4}

In this section, we consider the time decay rates for (\ref{eq1}). We first prove the propagation of low-frequency regularity for the solutions, which is crucial for the proof of Theorem \ref{theo2}.
\begin{lemm}\label{lemma4.1}
Let $\frac{1}{2}\le\beta<1$. Let $(a,u)$ be the global solution under the assumptions of Theorem \ref{theo1}. If additionally $(a_0,u_0)^{l}\in \dot B_{2,\infty}^{-\sigma}$ with $2\beta-\frac{3}{2}<\sigma\le\frac{1}{2}$ and
\begin{align*}
\|(a_0,u_0)\|^{l}_{\dot B_{2,\infty}^{-\sigma}}\le C,
\end{align*}
 then we have
\begin{align*}
\|(a,u)\|^{l}_{\dot B_{2,\infty}^{-\sigma}}\le C.
\end{align*}
\end{lemm}

\begin{proof}
Standard $L^2$ energy estimates give
\begin{align*}
\frac{1}{2}\frac{d}{dt}\|(a_j,u_j)\|_{L^2}^{2}+\|\Lambda^{\beta}u_j\|_{L^{2}}^{2}=\langle F_j,a_j\rangle+\langle H_j,u_j\rangle,
\end{align*}
It follows from a similar derivation of (\ref{307}) that
\begin{align*}
\left(\|(a,u)\|_{\dot B^{-\sigma}_{2,\infty}}^{l}\right)^{2}\le C\left(\|(a_0,u_0)\|_{\dot B^{-\sigma}_{2,\infty}}^{l}\right)^{2}+\int_{0}^{t}\|(F,H)\|_{\dot B^{-\sigma}_{2,\infty}}^{l}\|(a,u)\|_{\dot B^{-\sigma}_{2,\infty}}^{l}dt^{\prime}.
\end{align*}
Let
$$
E_{0}(t)=\|(a,u)\|_{\dot B^{-\sigma}_{2,\infty}}^{l}.
$$
By Lemma \ref{lemma2.5}, we conclude that for any $2\beta-\frac{3}{2}<\sigma\le\frac{1}{2}$, there holds
\begin{align}\label{401}
\dot B^{-\sigma}_{2,\infty}\times \dot B^{\frac{1}{2}}_{2,1}\longrightarrow \dot B^{-\sigma}_{2,\infty}.
\end{align}
We next estimate the nonlinear terms. For $F$, due to the above product law, we conclude that
\begin{align*}
\|a_xu\|^l_{\dot B^{-\sigma}_{2,\infty}}
\le &C\left(\|a_x^lu^l\|_{\dot B^{-\sigma}_{2,\infty}}+\|a^l_xu^h\|_{\dot B^{-\sigma}_{2,\infty}}+\|a_x^hu\|_{\dot B^{-\frac{1}{2}}_{2,\infty}}\right)\\
\le &C\left(\|a_x^l\|_{\dot B^{\frac{1}{2}}_{2,1}}\|u^l\|_{\dot B^{-\sigma}_{2,\infty}}+\|a_x^l\|_{\dot B^{\frac{1}{2}}_{2,1}}\|u^h\|_{\dot B^{-\sigma}_{2,\infty}}+\|a_x^h\|_{\dot B^{-\frac{1}{2}}_{2,\infty}}\|u\|_{\dot B^{\frac{1}{2}}_{2,1}}\right)\\
\le &CE_{0}D+CED+\|a_x^h\|_{\dot B^{-\frac{1}{2}}_{2,\infty}}\|u\|_{\dot B^{\frac{1}{2}}_{2,1}}.
\end{align*}
For the other terms, we also have
\begin{align*}
\|au_x\|^l_{\dot B^{-\sigma}_{2,\infty}}
\le C\|u\|_{\dot B^{\frac{3}{2}}_{2,1}}\left(\|a^l\|_{\dot B^{-\sigma}_{2,\infty}}+\|a^h\|_{\dot B^{-\sigma}_{2,\infty}}\right)\le CDE_{0}+CDE.
\end{align*}
For $H$, 
Applying (\ref{401}) and Lemma \ref{lemma2.6} again leads to
\begin{align*}
\|a_xK(a)\|^l_{\dot B^{-\sigma}_{2,\infty}}
\le &C\left(\|a_x^lK(a)\|_{\dot B^{-\sigma}_{2,\infty}}+\|a^h_xK(a)\|_{\dot B^{-\sigma}_{2,\infty}}\right)\\
\le &C\left(\|a_x^l\|_{\dot B^{\frac{1}{2}}_{2,1}}\|a\|^{l}_{\dot B^{-\sigma}_{2,\infty}}+\|a_x^l\|_{\dot B^{\frac{1}{2}}_{2,1}}\|a\|^{h}_{\dot B^{-\sigma}_{2,\infty}}+\|a^h\|_{\dot B^{1-\sigma}_{2,\infty}}\|a\|_{\dot B^{\frac{1}{2}}_{2,1}}\right)\\
\le &CDE_{0}+CDE,
\end{align*}
and
\begin{align*}
\|uu_x\|^l_{\dot B^{-\sigma}_{2,\infty}}
\le &C\left(\|u^lu_{x}\|_{\dot B^{-\sigma}_{2,\infty}}+\|u^hu_{x}\|_{\dot B^{-\sigma}_{2,\infty}}\right)\\
\le &C\left(\|u_x\|_{\dot B^{\frac{1}{2}}_{2,1}}\|u^l\|_{\dot B^{-\sigma}_{2,\infty}}+\|u_x\|_{\dot B^{\frac{1}{2}}_{2,1}}\|u^h\|_{\dot B^{-\sigma}_{2,\infty}}\right)\\
\le &CDE_{0}+CDE,
\end{align*}
and
\begin{align*}
\|G(a)(-\Delta)^{\beta}u\|_{\dot B^{-\sigma}_{2,\infty}}^l
&\le C\left(\|a\|_{\dot B^{-\sigma}_{2,\infty}}\|u^l\|_{\dot B^{\frac{1}{2}+2\beta}_{2,\infty}}+\|a\|_{\dot B^{\frac{1}{2}}_{2,1}}\|u^{h}\|_{\dot B^{\frac{3}{2}}_{2,1}}\right)\\
&\le CDE_{0}+CDE.
\end{align*}
Combining the above, we have
\begin{align*}
E_0^{2}(t)\le CE_{0}^{2}(0)+C\int_{0}^{t}\left(E_{0}D+ED+\|a_x^h\|_{\dot B^{-\frac{1}{2}}_{2,\infty}}\|u\|_{\dot B^{\frac{1}{2}}_{2,1}}\right)E_{0}dt^{\prime}.
\end{align*}
Thanks to Lemma \ref{lemma2.3} and (\ref{324}), we obtain
\begin{align*}
\int_{0}^{t}\left(\|a_x^h\|_{\dot B^{-\frac{1}{2}}_{2,\infty}}\|u\|_{\dot B^{\frac{1}{2}}_{2,1}}\right)dt^{\prime}\le &C\|a^h\|_{L^{2\beta}\left(\dot B^{\frac{1}{2}}_{2,1}\right)}\|u\|_{L^{\frac{2\beta}{2\beta-1}}\left(\dot B^{\frac{1}{2}}_{2,1}\right)}\\
\le &C\|a^h\|_{\widetilde L^{2\beta}\left(\dot B^{\frac{1}{2}}_{2,1}\right)}\|u\|_{\widetilde L^{\frac{2\beta}{2\beta-1}}\left(\dot B^{\frac{1}{2}}_{2,1}\right)}<\infty,
\end{align*}
where we used the fact that $a^h$ is in $\widetilde L^{\infty}\left(\dot B^{\frac{1}{2}}_{2,1}\right)\cap\widetilde L^{1}\left(\dot B^{\frac{1}{2}}_{2,1}\right)$. Combining the above two lines gives rise to
\begin{align}
E_0(t)\le C.
\end{align}
This proves Lemma \ref{lemma4.1}.
\end{proof}

\textbf{Proof of Theorem \ref{theo2}:}

By interpolation between Besov spaces, we obtain
\begin{align}\label{420}
\|(a,u)^l\|_{\dot B^{\frac{3}{2}-2\beta}_{2,1}}\le C \left(\|(a,u)\|_{\dot B^{-\sigma}_{2,\infty}}^{l}\right)^{\frac{4\beta}{3+2\sigma}}\left(\|(a,u)^l\|_{\dot B^{\frac{3}{2}}_{2,1}}\right)^{\frac{3+2\sigma-4\beta}{3+2\sigma}}.
\end{align}
By Lemma \ref{lemma4.1}, we derive that
\begin{align}\label{421}
\|(a,u)^l\|_{\dot B^{\frac{3}{2}}_{2,1}}\ge C\left(\|(a,u)^l\|_{\dot B^{\frac{3}{2}-2\beta}_{2,1}}\right)^{\frac{3+2\sigma}{3+2\sigma-4\beta}}.
\end{align}
On the other hand, due to Theorem \ref{theo1}, we have
\begin{align}\label{422}
\|\Lambda a\|^{h}_{\dot B^{\frac{3}{2}-2\beta}_{2,1}}\ge C\left(\|a\|^{h}_{\dot B^{\frac{1}{2}}_{2,1}}\right)^{\frac{3+2\sigma}{3+2\sigma-4\beta}},\quad \|u\|^{h}_{\dot B^{\frac{3}{2}}_{2,1}}\ge C\left(\|u\|^{h}_{\dot B^{\frac{3}{2}-2\beta}_{2,1}}\right)^{\frac{3+2\sigma}{3+2\sigma-4\beta}}.
\end{align}
Combining (\ref{420})-(\ref{422}), we have
\begin{align*}
\frac{d}{dt}E+CE^{\frac{3+2\sigma}{3+2\sigma-4\beta}}\le 0.
\end{align*}
A direct calculation yields
\begin{align}
E\le C(1+t)^{-\frac{3+2\sigma-4\beta}{4\beta}}.
\end{align}
Hence, we deduce that
\begin{align}
\|(a,u)\|_{\dot B^{\frac{3}{2}-2\beta}_{2,1}}\le C(1+t)^{-\frac{3+2\sigma-4\beta}{4\beta}}.
\end{align}
Therefore, for any $-\sigma< s<\frac{3}{2}-2\beta$, 
\begin{align}\label{423}
\|(a,u)\|_{\dot B^{s}_{2,1}}\le &C\left(\|(a,u)\|_{\dot B^{\frac{3}{2}-2\beta}_{2,\infty}}\right)^{\frac{2s+2\sigma}{3-4\beta+2\sigma}}\left(\|(a,u)\|_{\dot B^{-\sigma}_{2,\infty}}\right)^{\frac{3-4\beta-2s}{3-4\beta+2\sigma}}\\\nonumber
\le &C(1+t)^{-\frac{s+\sigma}{2\beta}}\nonumber.
\end{align}
Moreover, we obtain
\begin{align}
\|\Lambda^s(a,u)\|_{L^2}\le C(1+t)^{-\frac{s+\sigma}{2\beta}}\nonumber.
\end{align}
Next, we want to derive the time decay rates for $\|\Lambda^{s} a\|_{L^2}$ with $\frac{3}{2}-2\beta<s\le\frac{1}{2}$. Multiplying (\ref{306}) by $2^{\frac{1}{2}j}$ and taking the $l^{1}$ norm, we have
\begin{align}\label{424}
\frac{d}{dt}\|(a,u)^l\|_{\dot B^{\frac{1}{2}}_{2,1}}+c_{0}\|(a,u)^l\|_{\dot B^{\frac{1}{2}+2\beta}_{2,1}}\le C\|(F,H)^{l}\|_{\dot B^{\frac{1}{2}}_{2,1}}.
\end{align}
Combining (\ref{317}) and (\ref{424}), we have
\begin{align}\label{425}
&\frac{d}{dt}\left[\underbrace{\|(a,u)^l\|_{\dot B^{\frac{1}{2}}_{2,1}}+\|(\Lambda^{2\beta-1}a,u)\|^{h}_{\dot B^{\frac{3}{2}-2\beta}_{2,1}}}_{\bar E(t)}\right]+c_{0}\left(\underbrace{\|\Lambda a\|^{h}_{\dot B^{\frac{3}{2}-2\beta}_{2,1}}+\|u\|^{h}_{\dot B^{\frac{3}{2}}_{2,1}}+\|(a,u)^l\|_{\dot B^{\frac{1}{2}+2\beta}_{2,1}}}_{\bar D(t)}\right)\\\nonumber
\le\ & C\epsilon\|u_x\|_{\dot B^{\frac{1}{2}}_{2,1}}\|a\|_{\dot B^{\frac{1}{2}}_{2,1}}+\|F\|^{h}_{\dot B^{\frac{5}{2}-4\beta}_{2,1}}+\|H\|^{h}_{\dot B^{\frac{3}{2}-2\beta}_{2,1}}+C\|(F,H)^{l}\|_{\dot B^{\frac{1}{2}}_{2,1}}.
\end{align}
Next, we estimate the nonlinear terms. Thanks to the fact that $\dot B^{\frac{1}{2}}_{2,1}$ is an algebra and (\ref{3181}), we have
\begin{align}\label{426}
\|F^{l}\|_{\dot B^{\frac{1}{2}}_{2,1}}&\le
C\|au^h\|_{\dot B^{\frac{1}{2}}_{2,1}}+C\|au^l_{x}\|_{\dot B^{\frac{1}{2}}_{2,1}}+C\|a_{x}^{l}u^l\|_{\dot B^{\frac{1}{2}}_{2,1}}+C\|a_{x}^{h}u^l\|_{\dot B^{\frac{3}{2}-2\beta}_{2,1}}\\\nonumber
&\le C\|a\|_{\dot B^{\frac{1}{2}}_{2,1}}\|u^h\|_{\dot B^{\frac{1}{2}}_{2,1}}
+C\|a\|_{\dot B^{\frac{1}{2}}_{2,1}}\|u_{x}^l\|_{\dot B^{\frac{1}{2}}_{2,1}}+\|a_x^l\|_{\dot B^{\frac{1}{2}}_{2,1}}\|u^l\|_{\dot B^{\frac{1}{2}}_{2,1}}+\|a_x^h\|_{\dot B^{\frac{3}{2}-2\beta}_{2,1}}\|u^l\|_{\dot B^{\frac{1}{2}}_{2,1}}\\\nonumber
&\le C\bar ED.\nonumber
\end{align}
We will bound the term $F$ at high-frequencies. When $\frac{1}{2}<\beta<\frac{3}{4}$, from (\ref{31811})-(\ref{31812}) and (\ref{31814}), we know that
\begin{align}\label{427}
\|au_x\|^{h}_{\dot B^{\frac{5}{2}-4\beta}_{2,1}}&+\|a_x^lu^l\|^{h}_{\dot B^{\frac{5}{2}-4\beta}_{2,1}}+\|a_x^lu^h\|^{h}_{\dot B^{\frac{5}{2}-4\beta}_{2,1}}\\
&\le C\|a\|_{\dot B^{\frac{1}{2}}_{2,1}}\|u_x\|_{\dot B^{\frac{1}{2}}_{2,1}}+
C\|a_x^l\|_{\dot B^{\frac{1}{2}}_{2,1}}\|u^l\|_{\dot B^{\frac{1}{2}}_{2,1}}+C\|a_x^l\|_{\dot B^{\frac{1}{2}}_{2,1}}\|u^h\|_{\dot B^{\frac{5}{2}-4\beta}_{2,1}}\le C\bar ED.\nonumber
\end{align}
Thanks to $\frac{1}{2}<\beta<\frac{3}{4}$, by (\ref{3181}) and (\ref{318140}), we have
\begin{align}\label{428}
\|a_x^hu\|^{h}_{\dot B^{\frac{5}{2}-4\beta}_{2,1}}\le & C\|a_x^hu^{l}\|^{h}_{\dot B^{\frac{3}{2}-2\beta}_{2,1}}+C\|a_x^hu^h\|^{h}_{\dot B^{\frac{5}{2}-4\beta}_{2,1}}\\\nonumber
\le &C\|a_x^h\|_{\dot B^{\frac{3}{2}-2\beta}_{2,1}}(\|u^l\|_{\dot B^{\frac{1}{2}}_{2,1}}+\|u^h\|_{\dot B^{\frac{3}{2}-2\beta}_{2,1}})\le C\bar ED,\nonumber
\end{align}
Combining (\ref{427}) and (\ref{428}), when $\frac{1}{2}<\beta<\frac{3}{4}$, there holds
\begin{align}\label{429}
\|F\|^{h}_{\dot B^{\frac{5}{2}-4\beta}_{2,1}} \le C\bar ED.
\end{align}
When $\frac{3}{4}\le\beta<1$, by virtue of (\ref{31910}), (\ref{31816})-(\ref{31817}), we have
\begin{align}\label{430}
&\|a^h u\|^{h}_{\dot B^{\frac{7}{2}-4\beta}_{2,1}}+\|a^l u^h\|^{h}_{\dot B^{\frac{7}{2}-4\beta}_{2,1}}+\|a^l u^l\|^{h}_{\dot B^{\frac{7}{2}-4\beta}_{2,1}}\\\nonumber
\le\ &C\|a^h\|_{\dot B^{\frac{7}{2}-4\beta}_{2,1}}\|u^l\|_{\dot B^{\frac{1}{2}}_{2,1}}+\|a^h\|_{\dot B^{\frac{1}{2}}_{2,1}}\|u^h\|_{\dot B^{\frac{7}{2}-4\beta}_{2,1}}+C\|a^l\|_{\dot B^{\frac{1}{2}}_{2,1}}\|u^h\|_{\dot B^{\frac{7}{2}-4\beta}_{2,1}}\\
&+C\|a^l\|_{\dot B^{\frac{3}{2}}_{2,1}}\|u^l\|_{L^{\infty}}+C\|u^l\|_{\dot B^{\frac{3}{2}}_{2,1}}\|a^l\|_{L^{\infty}}\le C\bar ED.\nonumber
\end{align}
Then we get, if $\frac{3}{4}\le\beta<1$, that
\begin{align}\label{431}
\|F\|^{h}_{\dot B^{\frac{5}{2}-4\beta}_{2,1}} \le C\bar ED.
\end{align}
For $H$, we first deal with the term $K(a)a_x$. Due to (\ref{3201}), we have
 \begin{align}\label{433}
\|K(a)a_x\|^{l}_{\dot B^{\frac{1}{2}}_{2,1}}+\|K(a)a_x\|^{h}_{\dot B^{\frac{3}{2}-2\beta}_{2,1}}&\le C\|K(a)a_x^{l}\|_{\dot B^{\frac{1}{2}}_{2,1}}+C\|K(a)a_x^{h}\|_{\dot B^{\frac{3}{2}-2\beta}_{2,1}}\\\nonumber
&\le C\|a\|_{\dot B^{\frac{1}{2}}_{2,1}}\|a_x^l\|_{\dot B^{\frac{1}{2}}_{2,1}}+C\|a\|_{\dot B^{\frac{1}{2}}_{2,1}}\|a^h_x\|_{\dot B^{\frac{3}{2}-2\beta}_{2,1}}\\\nonumber
&\le C\bar ED.
 \end{align}
For the term $uu_x$, thanks to (\ref{3181}) again, we have
\begin{align}\label{444}
\|uu_x\|^{l}_{\dot B^{\frac{1}{2}}_{2,1}}+\|uu_x\|^{h}_{\dot B^{\frac{3}{2}-2\beta}_{2,1}}&\le C\|u^lu_x\|_{\dot B^{\frac{1}{2}}_{2,1}}+C\|u^hu_x\|_{\dot B^{\frac{3}{2}-2\beta}_{2,1}}\\\nonumber
&\le C\|u^l\|_{\dot B^{\frac{1}{2}}_{2,1}}\|u_x\|_{\dot B^{\frac{1}{2}}_{2,1}}+C\|u^h\|_{\dot B^{\frac{3}{2}-2\beta}_{2,1}}\|u_x\|_{\dot B^{\frac{1}{2}}_{2,1}}\\\nonumber
&\le C\bar ED.
\end{align}
The remaining term in $H$ can be handled as follows
\begin{align}\label{445}
\left\|G(a)(-\Delta)^{\beta}u\right\|^{l}_{\dot B^{\frac{1}{2}}_{2,1}}+\left\|G(a)(-\Delta)^{\beta}u\right\|^{h}_{\dot B^{\frac{3}{2}-2\beta}_{2,1}}\le C\|a\|_{\dot B^{\frac{1}{2}}_{2,1}}\|u\|_{\dot B^{\frac{3}{2}}_{2,1}}\le C\bar ED.
\end{align}
Plugging (\ref{426})-(\ref{445}) into (\ref{425}), we get
\begin{align}\label{446}
\frac{d}{dt}\bar E+C\bar D\le C\bar ED.
\end{align}
By virtue of Lemma \ref{lemma4.1} and interpolation between Besov spaces, we have
\begin{align}\label{447}
\|(a,u)^l\|_{\dot B^{\frac{1}{2}}_{2,1}}\le C \left(\|(a,u)\|_{\dot B^{-\sigma}_{2,\infty}}^{l}\right)^{\frac{4\beta}{1+2\sigma+4\beta}}\left(\|(a,u)^l\|_{\dot B^{\frac{1}{2}+2\beta}_{2,1}}\right)^{\frac{1+2\sigma}{1+2\sigma+4\beta}}.
\end{align}
By Lemma \ref{lemma4.1} again, we deduce that
\begin{align}\label{448}
\|(a,u)^l\|_{\dot B^{\frac{1}{2}+2\beta}_{2,1}}\ge C\left(\|(a,u)^l\|_{\dot B^{\frac{1}{2}}_{2,1}}\right)^{\frac{1+2\sigma+4\beta}{1+2\sigma}}.
\end{align}
On the other hand, due to Theorem \ref{theo1}, we have
\begin{align}\label{449}
\|\Lambda a\|^{h}_{\dot B^{\frac{3}{2}-2\beta}_{2,1}}\ge C\left(\|a\|^{h}_{\dot B^{\frac{1}{2}}_{2,1}}\right)^{\frac{1+2\sigma+4\beta}{1+2\sigma}},\quad \|u\|^{h}_{\dot B^{\frac{3}{2}}_{2,1}}\ge C\left(\|u\|^{h}_{\dot B^{\frac{3}{2}-2\beta}_{2,1}}\right)^{\frac{1+2\sigma+4\beta}{1+2\sigma}}.
\end{align}
Therefore, we conclude that
\begin{align*}
\frac{d}{dt}\bar E+C\bar E^{\frac{1+2\sigma+4\beta}{1+2\sigma}}\le C\bar ED.
\end{align*}
Then we derive
\begin{align*}
\frac{d}{dt}\left(\mathrm{e}^{-C\int_{0}^{t}Ddt}\bar E\right)+C\left(\mathrm{e}^{-C\int_{0}^{t}Ddt}\bar E^{\frac{1+2\sigma+4\beta}{1+2\sigma}}\right)\le 0.
\end{align*}
Noting that 
\begin{align*}
C\le\mathrm{e}^{-C\int_{0}^{t}Ddt}\le1.
\end{align*}
Then we get
\begin{align*}
\frac{d}{dt}\left(\mathrm{e}^{-C\int_{0}^{t}Ddt}\bar E\right)+C\left(\mathrm{e}^{-C\int_{0}^{t}Ddt}\bar E\right)^{\frac{1+2\sigma+4\beta}{1+2\sigma}}\le 0.
\end{align*}
A direct calculation yields
\begin{align*}
\bar E\le C(1+t)^{-\frac{1+2\sigma}{4\beta}}.
\end{align*}
Hence, we deduce that
\begin{align*}
\|a\|_{\dot B^{\frac{1}{2}}_{2,1}}\le C(1+t)^{-\frac{1+2\sigma}{4\beta}}.
\end{align*}
Therefore, for any $-\sigma< s<\frac{1}{2}$, 
\begin{align*}
\|a\|_{\dot B^{s}_{2,1}}\le &C\left(\|a\|_{\dot B^{\frac{1}{2}}_{2,\infty}}\right)^{\frac{2\sigma+2s}{1+2\sigma}}\left(\|a\|_{\dot B^{-\sigma}_{2,\infty}}\right)^{\frac{1-2s}{1+2\sigma}}\le C(1+t)^{-\frac{s+\sigma}{2\beta}}\nonumber.
\end{align*}
Moreover, we obtain
\begin{align*}
\|\Lambda^s a\|_{L^2}\le C(1+t)^{-\frac{s+\sigma}{2\beta}}\nonumber.
\end{align*}
As a consequence, we complete the proof of Theorem \ref{theo2}.
\hfill\qedsymbol
 \vskip .1in
 \noindent{\bf Acknowledgement} This work was partially supported by the National Natural Science Foundation of
  China (No.12571261).

 \vskip .1in
\noindent{\bf Data Availability Statement} Data sharing is not applicable to this article as no
data sets were generated or analyzed during the current study.

\vskip .1in

\noindent{\bf Conflict of Interest} The authors declare that they have no conflict of interest. The
authors also declare that this manuscript has not been previously published, and
will not be submitted elsewhere before your decision.

\phantomsection
\bibliographystyle{abbrv}
\bibliography{1D}


  

\end{document}